\documentclass[a4paper,11pt,reqno]{amsart}

\usepackage{graphicx}

\newtheorem{theorem}{Theorem}[section]
\newtheorem{lemma}[theorem]{Lemma}
\newtheorem{proposition}[theorem]{Proposition}
\newtheorem{corollary}[theorem]{Corollary}

\theoremstyle{definition}

\newtheorem{definition}[theorem]{Definition}

\newtheorem*{example}{Example}

\newtheorem*{remark}{Remark}

\numberwithin{equation}{section}

\numberwithin{theorem}{section}

\begin{document}

\title{Tropical Nevanlinna theory and second main theorem}

\author[I. Laine]{Ilpo Laine}

\address{\sc University of Joensuu, Department of Mathematics
\newline P.O. Box 111, FI-80101 Joensuu, Finland}
\email{ilpo.laine@joensuu.fi}

\author[K. Tohge]{Kazuya Tohge}

\address{\sc College of Science and Engineering, Kanazawa University,
\newline Kakuma-machi, Kanazawa, 920-1192, Japan}
\email{tohge@t.kanazawa-u.ac.jp}

\thanks{The first author has been supported by the Academy of Finland
grant 124954}

\thanks{The second author has been supported by the Japan Society
for the Promotion of Science Grant-in-Aid for Scientific Research
(C) 19540173}

\subjclass[2000]{Primary 14; Secondary 30D}

\maketitle

\section{\textbf{Introduction}}\label{TropI}

\medskip

Tropical Nevanlinna theory, see \cite{HS}, describes value
distribution of continuous piecewise linear functions of a real
variable whose one-sided derivatives are integers at every point,
similarly as meromorphic functions are described in the classical
Nevanlinna theory \cite{CYe}, \cite{H}, \cite{L}. In this paper, we
take an extended point of view to tropical meromorphic functions by
dispensing with the requirement of integer one-sided derivatives.
Accepting that multiplicities of poles, resp. zeros, may be
arbitrary real numbers instead of being integers, resp. rationals,
as in the classical theory of (complex) meromorphic functions, resp.
of algebroid functions, it appears that previous results such as in
\cite{HS}, \cite{LY2}, continue to be valid, with slight
modifications only in the proofs.

\medskip

Recalling the standard one-dimensional tropical framework, we shall
consider a max-plus semi-ring endowing $\mathbb{R}\cup\{ -\infty\}$
with (tropical) addition
$$x\oplus y:=\max (x,y)$$
and (tropical) multiplication
$$x\otimes y:=x+y.$$
We also use the notations $x\oslash y:=x-y$ and $x^{\otimes\alpha}
:=\alpha x$, for $\alpha\in\mathbb{R}$. The identity elements for
the tropical operations are $0_{\circ}=-\infty$ for addition and
$1_{\circ}=0$ for multiplication. Observe that such a structure is
not a ring, since not all elements have tropical additive inverses.
For a general background concerning tropical mathematics, see
\cite{SS}.

\bigskip

Concerning meromorphic functions in the tropical setting, and their
elementary Nevanlinna theory, see the recent paper by Halburd and
Southall \cite{HS} as well as \cite{LY2} for certain additional
developments.

\begin{definition}\label{mero} A continuous
piecewise linear function $f:\mathbb{R}\rightarrow\mathbb{R}$ is
said to be tropical meromorphic.
\end{definition}

\textbf{Remarks.} (1) In \cite{HS} and \cite{LY2}, for a continuous
piecewise linear function $f:\mathbb{R}\rightarrow\mathbb{R}$ to be
tropical meromorphic, an additional requirement had been imposed
upon that both one-sided derivatives of $f$ were integers at each
point $x\in\mathbb{R}$. In the present paper, this additional
requirement has been removed. Indeed, the authors are greatful to
Prof. Aimo Hinkkanen for the idea of permitting real slopes in the
definition of tropical meromorphic functions. See also \cite{HS}, p.
900.

\medskip

(2) Observe that whenever $f:\mathbb{R}\rightarrow\mathbb{R}$ is a
continuous piecewise linear function, then the discontinuities of
$f'$, see below, have no limit points in $\mathbb{R}$.

\bigskip

A point $x$ of derivative discontinuity of a tropical meromorphic
function such that
$$\omega_{f}(x):=\lim_{\varepsilon\rightarrow 0+}(f'(x+\varepsilon )-f'(x-\varepsilon ))<0$$
is said to be a pole of $f$ of multiplicity $-\omega_{f}(x)$, while
if $\omega_{f}(x)>0$, then $x$ is called a root (or a zero-point) of
$f$ of multiplicity $\omega_{f}(x)$. Observe that the multiplicity
may be any real number, to be denoted as $\tau_{f}(x)$ in what
follows.

\medskip

The basic notions of the Nevanlinna theory are now easily set up
similarly as in \cite{HS}:

\medskip

The tropical proximity function for tropical meromorphic functions
is defined as
\begin{equation}\label{prox}
 m(r,f):=\frac{1}{2}(f^{+}(r)+f^{+}(-r)).
\end{equation}
Denoting by $n(r,f)$ the number of distinct poles of $f$ in the
interval $(-r,r)$, each pole multiplied by its multiplicity
$\tau_{f}$, the tropical counting function for the poles in $(-r,r)$
is defined as
\begin{equation}\label{count}
 N(r,f):=\frac{1}{2}\int_{0}^{r}n(t,f)dt=\frac{1}{2}\sum_{|b_{\nu}|<r}\tau_{f}(b_{\nu})(r-|b_{\nu}|).
\end{equation}
Defining then the tropical characteristic function $T(r,f)$ as
usual,
\begin{equation}\label{char}
 T(r,f):=m(r,f)+N(r,f),
\end{equation}
the tropical Poisson--Jensen formula, see \cite{HS}, p. 5--6, to be
proved below, readily implies the tropical Jensen formula
\begin{equation}\label{jensen}
T(r,f)-T(r,-f)=f(0)
\end{equation}
as a special case.

\medskip

In this paper, we first recall basic results of Nevanlinna theory
for tropical meromorphic functions, closely relying to what has been
made in \cite{HS} by Halburd and Southall. As a novel element, not
being included in \cite{HS}, we propose a result that might be
called as the tropical second main theorem.

\medskip

Next, for completeness, we recall tropical counterparts of three key
lemmas from Nevanlinna theory, frequently applied to complex
differential and difference equations, namely the Valiron--Mohon'ko
lemma, the Mohon'ko lemma and the Clunie lemma, see e.g.,
respectively, \cite{M1}, p. 83, \cite{C}, Lemma 2, and \cite{MM},
Theorem 6. As for the corresponding results in the tropical setting,
see \cite{LY2}. Indeed, the reader may easily verify that same
proofs as given in \cite{LY2}, carry over to the present situation
word by word.

\medskip

In the final part of the paper, we consider periodic tropical
meromorphic functions, a discrete version of the exponential
function and some ultra-discrete difference equations on the real
line as applications of the tropical Nevanlinna theory.

\section{Poisson--Jensen formula in the tropical setting}

In what follows in this paper, a meromorphic function $f$ is to be
understood in the sense of Definition \ref{mero}, unless otherwise
specified. We may also call $f$ to be restricted meromorphic,
whenever all of its one-sided derivatives (slopes) are integers.

\medskip

The Poisson--Jensen formula in the extended tropical setting is
formally as in the restricted meromorphic case, see \cite{HS}, Lemma
3.1. The same proof applies. For the convenience of the reader,
however, we recall a complete proof here.

\begin{theorem}\label{PJ} Suppose $f$ is a meromorphic function on
$[-r,r]$, for some $r>0$ and denote the distinct zeros, resp. poles,
of $f$ in this interval by $a_{\mu}$, resp. by $b_{\nu}$, with their
corresponding multiplicities $\tau_{f}$ attached. Then for any $x\in
(-r,r)$ we get the Poisson--Jensen formula

$$ f(x)=\frac{1}{2}(f(r)+f(-r))+\frac{x}{2r}(f(r)-f(-r))$$
$$-\frac{1}{2r}\sum_{|a_{\mu}|<r}\tau_{f}(a_{\mu})(r^{2}-|a_{\mu}-x|r-a_{\mu}x)+\frac{1}{2r}\sum_{|b_{\nu}|<r}\tau_{f}(b_{\nu})(r^{2}-|b_{\nu}-x|r-b_{\nu}x).
$$

In the particular case of $x=0$ we obtain the tropical Jensen
formula
$$f(0)=\frac{1}{2}(f(r)+f(-r))-\frac{1}{2}\sum_{|a_{\mu}|<r}\tau_{f}(a_{\mu})(r-|a_{\mu}|)+\frac{1}{2}\sum_{|b_{\nu}|<r}\tau_{f}(b_{\nu})(r-|b_{\nu}|).$$

\end{theorem}

\begin{proof} As in \cite{HS}, we define an increasing sequence $(c_{j}), j=-p,\ldots
,q$ in $(-r,r)$ in the following way. Let $c_{0}=x$, and let the
other points in this sequence be the points in $(-r,r)$ at which the
derivative of $f$ does not exist, i.e. $f$ has either a zero or a
pole at these points. Further, we denote by $m_{j}$ slopes of the
line segments in the graph of $f$. In particular, we define
$m_{j-1}:=\lim_{x\rightarrow c_{j}^{-}}f'(x)$ for $j=-p,\ldots ,0$,
resp. $m_{j+1}:=\lim_{x\rightarrow c_{j}^{+}}f'(x)$ for $j=0,\ldots
,q$. Elementary geometric observation implies
$$f(r)-f(x)=m_{1}(c_{1}-x)+m_{2}(c_{2}-c_{1})+\cdots +m_{q}(c_{q}-c_{q-1})+m_{q+1}(r-c_{q})$$
$$=-m_{1}x+m_{q+1}r+c_{1}(m_{1}-m_{2})+\cdots +c_{q}(m_{q}-m_{q+1})$$
$$=m_{1}(r-x)-\sum_{j=1}^{q}(m_{j}-m_{j+1})(r-c_{j}).$$
By a parallel reasoning,
$$f(x)-f(-r)=m_{-1}(r+x)-\sum_{j=1}^{p}(m_{-j-1}-m_{-j})(r+c_{-j}).$$
Multiplying the above two equalities by $(r+x)$ and $(r-x)$,
respectively, and subtracting, we obtain
$$2rf(x)=r(f(r)+f(-r))+x(f(r)-f(-r))+(m_{-1}-m_{1})(r^{2}-x^{2})$$
$$+\sum_{j=1}^{p}(m_{-j-1}-m_{-j})(r^{2}-(x-c_{-j})r-c_{-j}x)$$
$$+\sum_{j=1}^{q}(m_{j}-m_{j+1})(r^{2}-(c_{j}-x)r-c_{j}x)=$$
$$=r(f(r)+f(-r))+x(f(r)-f(-r))+\sum_{c_{j}}-\omega_{f}(c_{j})(r^{2}-|c_{j}-x|r-c_{j}x).$$

Recalling the definition of the multiplicity $\tau_{f}$ for roots
and poles of $f$, the claim is an immediate consequence of this
equality.
\end{proof}

\section{Basic Nevanlinna theory in the tropical
setting}\label{Nlinna}

It is easy to verify that several basic inequalities, see \cite{HS},
for the proximity function and the characteristic function hold in
our present setting as well. In particular, the following simple
observations are immediately proved by the corresponding
definitions:

\begin{lemma}\label{elem} (i) If $f\leq g$, then $m(r,f)\leq
m(r,g)$.

\medskip

(ii) Given a real number $\alpha$, then
$$m(r,f^{\otimes\alpha})=m(r,\alpha f)=\alpha m(r,f),$$
$$N(r,f^{\otimes\alpha})=N(r,\alpha f)=\alpha N(r,f),$$
$$T(r,f^{\otimes\alpha})=T(r,\alpha f)=\alpha T(r,f).$$

\medskip

(iii) Given tropical meromorphic functions $f,g$, then
$$m(r,f\otimes g)\leq m(r,f)+m(r,g),$$
$$N(r,f\otimes g)\leq N(r,f)+N(r,g),$$
$$T(r,f\otimes g)\leq T(r,f)+T(r,g).$$
\end{lemma}

\textbf{Remark.} Observe that whenever $f\leq g$, the inequality
$N(r,f)\leq N(r,g)$ is not necessarily true. Similarly, the
inequality
$$N(r,f\oplus g)=N(r,\max (f,g))\leq\max (N(r,f),N(r,g))$$ may fail.
Indeed, as for the case $f\leq g$, take $f,g$ satisfying this
inequality so that the graph of $f$ is constant outside of $[-1,1]$
and is $\wedge\wedge$-shaped in $[-1,1]$, and let $g$ be defined
correspondingly as $\wedge$-shaped. Then $f$ has two poles, while
$g$ has only one. If the slopes are suitably defined, then
$N(r,f)>N(r,g)$. As for the case of $\max (f,g)$, a corresponding
example is easily constructed. The corresponding observations are
true for the characteristic function as well, provided just that the
proximity functions are small enough.

\bigskip

As usual in the Nevanlinna theory, the next step from the
Poisson--Jensen formula is to formulate the first main theorem. To
this end, we recall the notation $L_{f}:=\inf\{ f(b)\}$ over all
poles $b$ of $f$, i.e.
$$L_{f}:=\inf\{ f(b) : \omega_{f}(b)<0\}.$$
In particular, if $f$ has no poles (and so $f$ is said to be
tropical entire), then we have $L_{f}=\inf\emptyset =+\infty$.

\begin{theorem}\label{FMT} Let $f$ be tropical meromorphic. Then
$$T(r,1_{\circ}\oslash (f\oplus a))=T(r,-\max (f,a))\leq T(r,f)+\max (a,0)-\max (f(0),a)$$
for any $a\in\mathbb{R}$ and any $r>0$. Moreover, an asymptotic
equality
$$T(r,1_{\circ}\oslash (f\oplus a))=T(r,-\max (f,a))=T(r,f)-\max (f(0),a)+\varepsilon (r,a)$$
holds for any $r>0$ with $0\leq\varepsilon (r,a)\leq\max (a,0)$,
provided that $-\infty <a<L_{f}$.
\end{theorem}

\begin{proof} Making use of the tropical Jensen formula
(\ref{jensen}), we immediately conclude that
\begin{eqnarray*}
T\bigl(r, 1_0\oslash (f\oplus a)\bigr)&=&T\bigl(r, -\max(f, a)\bigr)\\
&=&T\bigl(r, \max(f, a)\bigr)-\max\bigl(f(0), a\bigr)\\
&\leq& T(r,f)+\max(a,0)-\max\bigl(f(0), a\bigr)
\end{eqnarray*}
for any $a\in \mathbb{R}$ and for any $r>0$. Here we also used the
inequality $T(r, g\oplus h)=T\bigl(r, \max(g,h)\bigr)\leq
T(r,g)+T(r,h)$ and the simple observation that $T(r,a)=\max(a,0)$.
Further,
\[
\max(a,0)-\max\bigl(f(0), a\bigr)= \left\{
\begin{array}{llc}
a-f(0)&\leq 0;\quad a>0, f(0)\geq a\\
a-a& =0; \quad a>0, f(0)< a\\
0-f(0)&\leq |a|; \quad a\leq 0, f(0)\geq a\\
0-a & =|a|; \quad a\leq 0, f(0)< a
\end{array}
\right\} \leq |a|.
\]

\medskip

To obtain the asserted asymptotic equality, suppose first that $f$
has at least one pole and that $-\infty <a<L_{f}$. In this case, we
have $N\bigl(r,\max(f,a)\bigr)=N(r,f)$. Therefore,
\begin{eqnarray*}
T\bigl(r, 1_0\oslash (f\oplus a)\bigr)&=&T\bigl(r, -\max(f, a)\bigr)\\
&=& T\bigl(r,\max(f, a)\bigr)-\max\bigl(f(0), a\bigr)\\
&=&m\bigl(r,\max(f, a)\bigr)+N\bigl(r,\max(f, a)\bigr)-\max\bigl(f(0), a\bigr)\\
&\geq & m(r, f)+N(r, f)-\max\bigl(f(0), a\bigr)\\
&\geq & T(r, f)-\max\bigl(f(0), a\bigr),
\end{eqnarray*}
according to the monotonicity of $m(r, \ast)$, Lemma \ref{elem},
with respect to the second component $\ast$.

\medskip

Finally, if $L_f=+\infty$, that is, if $f$ has no poles, the
asymptotic equality holds as well. In fact, because of
$T(r,f)=m(r,f)$ then and $f\oplus a\geq f$ for any $a\in\mathbb{R}$,
we have
\begin{eqnarray*}
T\bigl(r, 1_0\oslash (f\oplus a)\bigr)&=&T\bigl(r, -\max(f, a)\bigr)\\
&=& T\bigl(r,\max(f, a)\bigr)-\max\bigl(f(0), a\bigr)\\
&\geq&m\bigl(r,\max(f, a)\bigl)-\max\bigl(f(0), a\bigr)\\
&\geq & m(r, f)-\max\bigl(f(0), a\bigr)\\
&=& T(r, f)-\max\bigl(f(0), a\bigr).
\end{eqnarray*}
\end{proof}

\textbf{Example.} As an example, for a non-constant linear function
$f(x)=\alpha x+\beta$ with $\alpha
>0$ and $\beta >0$, say, it immediately follows that
\[
T(r,f)=m(r,f)=\left\{
\begin{array}{cl}
\beta & \bigl(0\leq r<\frac{\beta}{\alpha}\bigr)\\
\frac{\alpha}{2}r+\frac{\beta}{2} & \bigl(\frac{\beta}{\alpha}\leq
r\bigr)
\end{array}
\right.\,.
\]
It is a simple exercise to verify by this example that the error
term $\varepsilon(r,a)$ in Theorem \ref{FMT}, may run over the whole
interval $\biggl[0, \max\bigl(f(0),a\bigr)\biggr)$.

\bigskip

We next proceed to recall

\begin{theorem}\label{mono} The characteristic function $T(r,f)$ is
a positive, continuous, non-decreasing piecewise linear function of
$r$.
\end{theorem}

\begin{proof} The proof offered in
\cite{HS}, p. 894, applies verbatim.
\end{proof}

\textbf{Remark.} (1) The counting function $N(r,f)$ is a positive,
continuous, non-decreasing piecewise linear function of $r$ as well.

\medskip

(2) In particular, Theorem \ref{mono} and Remark (1) above imply
that standard Borel type theorems apply for $T(r,f)$ and $N(r,f)$,
see e.g. \cite{HS}, Lemma 3.5.

\medskip

(3) As a remark for further needs, the following estimate, see
\cite{HS}, remains valid in the present setting as well: Indeed, for
all $k>1$,
$$n(r,f)\leq\frac{2}{(k-1)r}N(kr,f).$$
Moreover, given $\varepsilon >0$, $R>0$ and combining this estimate
and a Borel type lemma, we get
$$n(r,f)\leq 4r^{-1}N(r,f)^{1+\varepsilon}$$
for all $r>R$ outside an exceptional set of finite logarithmic
measure, see \cite{HS}, Theorem 3.6.

\medskip

(4) Defining a tropical rational function as a meromorphic function
that has finitely many poles and zeros only, the first estimate in
(3) above may be used to show that a meromorphic function is
rational if and only if $T(r,f)=O(r)$, see \cite{HS}, Theorem 3.4.

\bigskip

Following the usual classical notion, a meromorphic function $f$ is
said to be of finite order of growth, if $T(r,f)\leq r^{\sigma}$ for
some positive number $\sigma$, and for all $r$ sufficiently large.
Of course, this enables us to define the order $\rho (f)$ of a
meromorphic function in the usual way as
$$\rho (f):=\limsup_{r\rightarrow\infty}\frac{\log T(r,f)}{\log r}.$$
In the finite order case, the characteristic function and the
counting function of the shifts of meromorphic functions may be
estimated by applying the following lemma, see \cite{K}, Lemma 3.2:

\begin{lemma}\label{shift} Let $T:[0,+\infty)\rightarrow [0,+\infty
)$ be a non-decreasing continuous function of finite order $\rho$
and take $c\in (0,+\infty )$. Then
$$T(r+c)=T(r)+O(r^{\rho -1+\varepsilon})$$
outside of a set of finite logarithmic measure.
\end{lemma}

Moreover, the estimates given in Remark (3) above, may be modified
in the finite order situation as follows, see \cite{HS}, Corollary
3.7:

\begin{lemma}\label{laskuri} Let $f$ be a meromorphic function of
finite order, and suppose that $\delta <1$ and $R>0$. Then
$n(r,f)\leq r^{-\delta}N(r,f)$ for all $r>R$ outside an exceptional
set of finite logarithmic measure.
\end{lemma}

In classical Nevanlinna theory and its applications, the lemma on
logarithmic derivatives plays a fundamental role. It is likely that
its tropical counterpart below, the lemma on tropical quotients of
shifts, may become equally important:

\begin{theorem}\label{LLD} Let $f$ be tropical meromorphic. Then,
for any $\varepsilon >0$,
$$m(r,f(x+c)\oslash f(x))\leq \frac{2^{1+\varepsilon}14|c|}{r}(T(r+|c|,f)^{1+\varepsilon}+o(T(r+|c|,f)))$$
holds outside an exceptional set of finite logarithmic measure.
\end{theorem}

\begin{proof} The proof given for Lemma 3.8 in \cite{HS}, see p. 897--898,
applies word by word.
\end{proof}

Another version of the lemma on tropical quotients of shifts is a
tropical counterpart of a discussion in \cite{HKT}:

\begin{lemma} \label{dlog}
Let $f$ be tropical meromorphic. Then for all $\alpha >1$ and $r>0$,
    \begin{equation*}
    m\bigl(r,f(x+c)\oslash f(x)\bigr) \leq
\frac{12|c|/(\alpha-1)}{r+|c|}\left\{T\big(\alpha(r+|c|),f\big)+f(0)/2\right\}\,.
    \end{equation*}
\end{lemma}

\begin{proof} Following \cite{HS} as in the proof of their Lemma~3.8, by
taking $\rho=(\alpha+1)(r+|c|)/2$ so that
$\rho-r-|c|=(\alpha-1)(r+|c|)/2$ and $\rho>r+|c|$, we have
$$
m\bigl(r, f(x+c)\oslash f(x)\bigr)\leq |c|\left\{\frac{\bigl(m(\rho,
f)+m(\rho,-f)\bigr)}{\rho}+\frac{3}{2}\bigl(n(\rho,f)+n(\rho,-f)\bigr)\right\}
$$
for $x\in [-r,r]$. Since
\begin{eqnarray*}
N\bigl(\alpha(r+|c|), \pm f\bigr)
&\geq & \frac{1}{2}\int_{(\alpha+1)(r+|c|)/2}^{\alpha(r+|c|)} n(t,\pm f)dt\\
&\geq & \frac{1}{2}n(\rho, \pm f )\frac{\alpha-1}{2}(r+|c|)\,,
\end{eqnarray*}
we get
$$
n(\rho, \pm f)\leq \frac{4}{\alpha-1}\frac{1}{r+|c|}
N\bigl(\alpha(r+|c|), \pm f\bigr) \leq
\frac{2}{\rho}\frac{\alpha+1}{\alpha-1}N\bigl(\alpha(r+|c|), \pm
f\bigr)\,,
$$
and therefore the desired estimate:
$$
m\bigl(r, f(x+c)\oslash f(x)\bigr)\leq
\frac{3|c|(\alpha+1)/(\alpha-1)}{\rho} \left\{2T(\rho,
f)+f(0)\right\}\,.
$$
\end{proof}

\begin{corollary}\label{LLDf} Let $f$ be a meromorphic function of
finite order $\rho$. Given $\varepsilon >0$, $f$ satisfies
$$m(r,f(x+c)\oslash f(x))=O(r^{\rho -1+\varepsilon})$$
outside an exceptional set of finite logarithmic measure.
\end{corollary}

\section{Tropical meromorphic functions of hyper-order less than one}

As pointed out in \cite{HKT}, a number of results in the difference
variant of the Nevanlinna theory, see \cite{HK}, typically expressed
for meromorphic functions of finite order, may also be formulated
for meromorphic functions of hyper-order less than one. This
extension applies to the tropical meromorphic setting as well. To
this end, first recall the definition of hyper-order
\begin{equation}\label{hyper}
\rho_{2}(f):=\limsup_{r\rightarrow\infty}\frac{\log\log T(r,f)}{\log
r}.
\end{equation}
Next recall the following lemma from \cite{HKT}, corresponding, in
the case of hyper-order less than one, to our previous Lemma 3.4:

\begin{lemma}\label{technical}
Let $T:[0,+\infty)\to[0,+\infty)$ be a non-decreasing continuous
function and let $s\in(0,\infty)$. If the hyper-order of $T$ is less
than one, i.e.,
    \begin{equation}\label{assu}
    \limsup_{r\to\infty}\frac{\log\log T(r)}{\log r}=\rho_{2}<1
    \end{equation}
and $\delta\in(0,1-\rho_{2})$ then
   \begin{equation}\label{concl}
    T(r+s) = T(r)+ o\left(\frac{T(r)}{r^{\delta}}\right)
    \end{equation}
where $r$ runs to infinity outside of a set of finite logarithmic
measure.
\end{lemma}

For a proof of this lemma, see \cite{HKT}.

\medskip

As a counterpart to Corollary \ref{LLDf}, we may state the following

\begin{proposition}\label{LLDh} Let $f$ be a meromorphic function of
hyper-order $\rho_{2}<1$. Given $\tau >1-\rho_{2}$, and fixing
$\delta\in (0,1-\rho_{2})$, then $f$ satisfies
$$m(r,f(x+c)\oslash f(x))\leq\frac{13|c|}{r^{\tau}}T(r,f)=o(T(r,f)/r^{\delta}),$$
as $r$ approaches to infinity outside of a set of finite logarithmic
measure.
\end{proposition}

\begin{proof} For the convenience of the reader, we give a complete proof
here, following an idea from Halburd and Korhonen [4]. First recall
a generalized Borel Lemma as given in \cite{CYe}, Lemma~3.3.1: Let
$\xi(x)$ and $\phi(s)$ be positive, nondecreasing and continuous
functions defined for all sufficiently large $x$ and $s$,
respectively, and let $C>1$. Then we have
    \begin{equation}\label{cy}
    T\left(s+\frac{\phi(s)}{\xi(T(s,f))},f\right) \leq C\,T(s,f)
    \end{equation}
for all $s$ outside of a set $E$ satisfying
    \begin{equation}\label{E}
    \int_{E\cap [s_0,R]} \frac{ds}{\phi(s)} \leq  \frac{1}{\log
     C}\int_e^{T(R,f)}\frac{dx}{x\xi(x)} +O(1)
    \end{equation}
where $R<\infty$. Since $f$ is of infinite order and of hyper-order
less than $1$, then by choosing $\phi(r)= r$, $\xi(x)=(\log
x)^{1+\varepsilon}$ for $\varepsilon>0$ and
    \begin{equation*}
    \alpha= 1+\frac{\phi(r+|c|)}{(r+|c|)\xi(T(r+|c|,f))},
    \end{equation*}
in the above estimate, it follows that
    \begin{equation}\label{m3}
    m\bigl(r, f(x+c)\oslash f(z) \bigr) \leq
    \frac{12|c|\bigl(\log T(r+|c|,f)\bigr)^{1+\varepsilon}} {r+|c|} \bigl\{C T(r+|c|,f) +f(0)/2 \bigr\}
    \end{equation}
as $r$ approaches infinity outside of an $r$-set of finite
logarithmic measure.

Taking now $\varepsilon$ small enough to satisfy
$(\rho_{2}+\varepsilon)(1+\varepsilon)<1$, we see that $\bigl(\log
T(r+|c|,f)\bigr)^{1+\varepsilon}/(r+|c|)=o(1)(r+|c|)^{-\tau}$ with
$\tau:=(1-\rho_{2}+\varepsilon)(1+\varepsilon)>0$ for all
sufficiently large $r$. Then Lemma~\ref{technical} and the above
estimate show that
$$
m\bigl(r,f(x+c)\oslash f(z)\bigr) \leq \frac{13|c|}{r^{\tau}}T(r,f)
=o(T(r,f)/r^{\delta})
$$
holds as $r$ approaches infinity outside of a set of finite
logarithmic measure.
\end{proof}

A `tropical exponential' function $e_{\alpha}(x)$ is found as a
solution to equation $y(x+1)=y(x)^{\otimes \alpha}$, see Section
\ref{tropexp} below for its definition and basic properties. This
function may be used to point out that the condition $\rho_{2}(f)<1$
cannot be dropped in general. In fact, we have for $\alpha>1$,
$$
m\bigl(r,e_{\alpha}(z+1)\oslash e_{\alpha}(z)\bigr)= (\alpha-1)T(r,
e_{\alpha})
$$
on the whole $\mathbb{R}$. Of course, Lemma~\ref{dlog} remains true
for $f(x)=e_{\alpha}(x)$ as well.

\section{Second main theorem in the tropical setting}

In this section, we offer a tropical counterpart to the second main
theorem. Observe, however, that the second main theorem in the
tropical setting may not be as complete as in the usual Nevanlinna
theory. This is due to the fact that certain elementary inequalities
in the classical Nevanlinna theory, in particular those for the
counting function, may fail in the tropical theory.

\begin{theorem} \label{thm:second}
Let $f$ be a tropical meromorphic function and put $L_f:=\inf\{ f(b)
\, : \, \omega_f(b)<0 \}$. Given $c>0$, $q\in\mathbb{N}$ and $q$
distinct values $a_j\in\mathbb{R}$ $(1\leq j\leq q)$ that satisfy
$\max(a_1, \ldots, a_q)< L_f$, then
\begin{multline}\label{secondest}
q T(r,f) \leq
\sum_{j=1}^q N\bigl(r, 1_{\circ}\oslash (f\oplus a_j)\bigr)+T\bigl(r, f(x+c)\bigr) \\
- N\bigl(r, 1_{\circ}\oslash f(x+c)\bigr) + m\bigl(r,f(x+c)\oslash f(x)\bigr) \\
- f(c) +(2q-1)\max_{1\leq j\leq q} \max(a_j, 0) + \sum_{j=1}^q \max
\bigl(f(0), a_j\bigr)
\end{multline}
holds all $r>0$.
\end{theorem}

Before proceeding to prove Theorem \ref{thm:second}, we define
\begin{equation}\label{mult}
N_{1}(r,f):=N\bigl(r,1_{\circ}\oslash
f(x+c)\bigr)+2N(r,f)-N\bigl(r,f(x+c)\bigr).
\end{equation}
Clearly, (\ref{mult}) is a tropical counterpart to the classical
counting function
$$N_{1}(r,f):=N(r,1/f')+2N(r,f)-N(r,f')$$
for multiple values of $f$ in the second main theorem for usual
meromorphic functions. Using (\ref{mult}), we may write
(\ref{secondest}) as
\begin{multline}
q T(r,f)-T\bigl(r, f(x+c)\bigr) \leq \sum_{j=1}^q N\bigl(r,
1_{\circ}\oslash (f\oplus a_j)\bigr) \\- N_{1}(r,f) +2 N(r,f)
-N\bigl(r,f(x+c)\bigr) + m\bigl(r,f(x+c)\oslash f(x)\bigr) +O(1).
\end{multline}
Suppose now that $f$ is of hyper-order $\rho_{2}<1$. Applying
Lemma~\ref{technical} to $T(r,f)$ and $N(r,f)$, and recalling
Proposition \ref{LLDh} (with $\tau >1-\rho_{2})$, we obtain

\begin{theorem}\label{SMT} Suppose $f$ is a nonconstant tropical meromorphic
function of hyper-order $\rho_{2}<1$, and take $0<\delta
<1-\rho_{2}$. If $q\geq 1$ distinct values $a_{1},\ldots
,a_{q}\in\mathbb{R}$ satisfy $\max (a_{1},\ldots ,a_{q})<L_{f}$,
then
\begin{equation}\label{SMTest}
(q-1)T(r,f)\leq\sum_{j=1}^q N\bigl(r, 1_{\circ}\oslash (f\oplus
a_j)\bigr)-N(r,1_{\circ}\oslash f) +o\bigl(T(r,f)/r^{\delta}\bigr)
\end{equation}
outside an exceptional set of finite logarithmic measure.
\end{theorem}

\begin{proof} The desired inequality immediately follows from Theorem \ref{thm:second},
combined with Proposition \ref{LLDh} and the next three
inequalities, each of them being valid outside an exceptional set of
finite logarithmic measure:
\begin{eqnarray*}
T\bigl(r, f(x+c)\bigr)
& \leq & T(r,f)+ N\bigl(r,f(x+c)\bigr) - N(r,f) +o\bigl(T(r,f)/r^{\delta}\bigr), \\
N\bigl(r, f(x+c)\bigr) & \leq &  N(r+|c|, f) = N(r, f)+
o\bigl(T(r,f)/r^{\delta}\bigr)
\end{eqnarray*}
and
$$N\bigl(r, 1_{\circ}\oslash f(x+c)\bigr)\geq N(r-|c|, 1_{\circ}\oslash f) =
N(r, 1_{\circ}\oslash f) +o\bigl(T(r,f)/r^{\delta}\bigr).
$$
These inequalities are immediate consequences of Lemma
\ref{technical}.
\end{proof}

\begin{remark} Observe that whenever $f$ is of finite order $\rho$,
and so of hyper-order $\rho_{2}=0$, the error term
$o\bigl(T(r,f)/r^{\delta}\bigr)$ in Theorem \ref{SMT} may be
replaced by $O(r^{\rho -1+\varepsilon})$ with $\varepsilon >0$.
Similarly in Corollary \ref{nodef} below.
\end{remark}

\begin{corollary}\label{nodef}
Suppose $f$ is a nonconstant tropical meromorphic function of
hyper-order $\rho_{2}<1$, and take $0<\delta <1-\rho_{2}$. If If
$q\geq 1$ distinct values $a_{1},\ldots ,a_{q}\in\mathbb{R}$ satisfy
$\max (a_{1},\ldots ,a_{q})<L_{f}$, and if $\ell_f:=\inf\{ f(a) \, :
\, \omega_f(a)>0 \}> -\infty$, then
\begin{equation} \label{SMT_f}
qT(r,f)\leq \sum_{j=1}^q N\bigl(r, 1_{\circ}\oslash (f\oplus
a_j)\bigr) + o\bigl(T(r,f)/r^{\delta}\bigr).
\end{equation}
outside an exceptional set of finite logarithmic measure. In
particular,
\begin{equation}\label{SMTall}
T(r,f)\leq N\bigl(r, 1_{\circ}\oslash (f\oplus a)\bigr)
+o\bigl(T(r,f)/r^{\delta}\bigr)
\end{equation}
holds for all $a\in\mathbb{R}$ such that $a<L_{f}$.
\end{corollary}

\begin{proof}
Let $a_{1},\ldots ,a_{q}$ be $q$ distinct real values such that
$a_j<L_f$, $1\leq j\leq q$. In order to prove the assertion
(\ref{SMT_f}), choose a real number $\mu$ such that
$$
\mu < \min\Bigl\{ \max(a_1, \ldots , a_q), \ell_f \Bigr\}
$$
and put
$$
g(x):=f(x)-\mu, \quad \tilde{a}_j:=a_j-\mu (>0) \ (1\leq j\leq q)
\quad \text{and} \quad \tilde{a}_0:=0.
$$
Then the following observations are easily checked:
\begin{itemize}
\item $\rho_2(g)=\rho_2(f)$ as well as $\rho (g)=\rho (f)$,
\item $\omega_g(x)\equiv \omega_f(x)$,
\item $L_g:=\inf\{ g(b) \, : \, \omega_g(b)<0 \}=L_f-\mu$,
\item $L_g-\max(\tilde{a}_0, \tilde{a}_1, \ldots, \tilde{a}_q)=L_f-\max(a_1, \ldots , a_q)>0$,
\item $\ell_g:=\inf\{ g(a) \, : \, \omega_g(a)>0 \}=\ell_f-\mu > 0$.
\end{itemize}
We now apply Theorem~\ref{SMT} to the function $g(x)$ and the $q+1$
distinct values $\tilde{a}_j$ to obtain
\begin{equation}\label{SMT_g}
qT(r,g)\leq\sum_{j=0}^q N\bigl(r, 1_{\circ}\oslash (g\oplus
\tilde{a}_j)\bigr) - N(r, 1_{\circ}\oslash g) +S_{\ast}(r,g)
\end{equation}
outside an exceptional set of finite logarithmic measure. Since
$\ell_g>0$, the two functions $1_{\circ}(g\oplus 0)=-\max(g,0)$ and
$1_{\circ}\oslash g=-g$ have exactly the same zeros, and therefore
$$
N\bigl(r, 1_{\circ}\oslash (g\oplus \tilde{a}_0)\bigr) - N(r,
1_{\circ}\oslash g) \equiv 0
$$
in the above inequality. Since $T(r, f-\mu)\geq T(r,f)-\mu$ and
$g\oplus \tilde{a}_j=(f\oplus a_j)-\mu$, we obtain the desired
estimate (\ref{SMT_f}) for the original function $f$.
\end{proof}

\begin{remark} We should perhaps point out here that by this corollary,
nonconstant tropical meromorphic functions of hyper-order
$\rho_{2}<1$ and satisfying $\ell_f\neq -\infty$ have no deficient
values $a<L(f)$ in the sense that
$$1-\limsup_{r\rightarrow\infty}\frac{N(r,1_{\circ}\oslash (f\oplus a))}{T(r,f)}=0.$$
However, omitted values may well appear. For example, any linear
function has two omitted values $0$ and $\infty$. Moreover, rational
functions of shape $\wedge$, resp. of $\vee$, also omit roots, resp.
poles. Indeed, the estimates (\ref{secondest}), (\ref{SMT_f}) and
(\ref{SMTall}) above don't include, in general, consideration of the
roots, resp. the poles. In fact, the counting functions
$N(r,1_{\circ}\oslash f)$ and $N(r,1_{\circ}\oslash (f\oplus 0))$ do
not coincide in general. Note that $m(r,1_{\circ}\oslash (f\oplus
0))=0$, since $1_{\circ}\oslash (f\oplus 0)=-f^{+}\leq 0$, hence
$N(r,1_{\circ}\oslash (f\oplus 0))=T(r,1_{\circ}\oslash (f\oplus
0))=T(r,f)-\max (f(0),0)$ by the first main theorem, Theorem
\ref{FMT}, for any tropical meromorphic function $f$. On the other
hand, to get $N(r,1_{\circ}\oslash f)=T(r,f)-\max (f(0),0)$, we need
to have $\ell_{f}\geq 0$.

\medskip

To illustrate these comments, consider the linear function
$f(x)=x+1$, taking $r$ large enough, and $q=1$, $a_1=0$ and $c>0$.
Then we have $L_f=\ell_f=+\infty$ and
$$
N(r,f)=N(r, 1_{\circ}\oslash
f)=N\bigl(r,f(x+c)\bigr)=N\bigl(r,1_{\circ}\oslash
f(x+c)\bigr)\equiv 0.
$$
Moreover, $T(r,f)=m(r,f)=\frac{r+1}{2}$,
$T\bigl(r,f(x+c)\bigr)=m\bigl(r,f(x+c)\bigr)=\frac{r+c+1}{2}$ and
$m\bigl(r, f(x+c)\oslash f(x)\bigr)=c$. Therefore, (\ref{secondest})
takes the form
$$\frac{r+1}{2}\leq r+\frac{c+1}{2},$$
while (\ref{SMT_f}) in Corollary \ref{nodef} becomes
$$\frac{r+1}{2}\leq \frac{r}{2} + O(r^{\varepsilon})$$
for any $\varepsilon>0$.

\medskip

Finally, we remark that the assumption $\rho_{2}<1$ above cannot be
deleted. To see this, the reader may consider the tropical
exponential functions $e_{\alpha}$ in Section \ref{tropexp} below.
In particular, such a function may have uncountably many deficient
values, see a remark after Proposition \ref{order}.
\end{remark}

\begin{remark} The estimate (\ref{secondest}) given in Theorem
\ref{thm:second} may indeed be written in the form
\begin{multline}
\sum_{k=1}^q m\Bigl(r, 1_{\circ}\oslash (f\
\oplus a_k)\Bigr) \\
\leq m\bigl(r, 1_{\circ}\oslash f(x+c)\bigr)
+ m\left(r, \bigoplus_{k=1}^q f(x+c)\oslash \bigl(f(x)\oplus a_k\bigr)\right)\\
+ (2q-1) \max_{1\leq k\leq q} \max(a_k, 0)+\sum_{k=1}^q \max (a_k,
0)\,.
\end{multline}
This is an obvious tropical counterpart to the classical inequality
$$
\sum_{k=1}^q m\biggl(r, 1/\bigl(f(z)-a_k\bigr)\biggr) \leq m\bigl(r,
1/f'(z)\bigr) + m\left(r,\sum_{k=1}^q
\frac{f'(z)}{f(z)-a_k}\right)+O(1)\,,
$$
see e.g. \cite{H}, 32--33. On the other hand, Theorem \ref{SMT} and
Corollary \ref{nodef} are reminiscent to the classical second main
theorem.
\end{remark}

\bigskip

In order to prove Theorem \ref{thm:second}, we prepare a sequence of
lemmas.

\begin{lemma} \label{lemma1}
For any $p\in\mathbb{N}$, any $a_k\in\mathbb{R}$ $(1\leq k\leq p)$
and any $c\in\mathbb{R}\setminus\{0\}$, we have
\begin{multline*}
m\biggl(r,1_{\circ}\oslash \bigl(\bigotimes_{k=1}^p (f\oplus
a_k)\bigr)\biggr)
\leq T\bigl(r, f(x+c)\bigr)- N\bigl(r, 1_{\circ}\oslash f(x+c)\bigr)+ \\
         +m\biggl(r,f(x+c) \oslash \bigl(\bigotimes_{k=1}^p (f\oplus a_k)\bigr)\biggr)-f(c)\,.
\end{multline*}
\end{lemma}

\begin{proof} Note that
$$
1_{\circ}\oslash \bigl(\bigotimes_{k=1}^p (f\oplus a_k)\bigr)
=\bigl(1_{\circ}\oslash f(x+c)\bigr)
      \otimes \biggl(f(x+c)\oslash \bigl(\bigotimes_{k=1}^p (f\oplus a_k)\bigr)\biggr).
$$
Since $m(r,g\otimes h)\leq m(r,g)+m(r,h)$, we have
\begin{eqnarray*}
& & \hspace{-1cm}
m\biggl(r,1_{\circ}\oslash \bigl(\bigotimes_{k=1}^p (f\oplus a_k)\bigr)\biggr) \\
&\leq& m\bigl(r, 1_{\circ}\oslash f(x+c) \bigr)
     +m\biggl(r, f(x+c) \oslash \bigl(\bigotimes_{k=1}^p (f\oplus a_k)\bigr)\biggr)\\
&=&
T\bigl(r, f(x+c))-f(0+c) - N\bigl(r, 1_{\circ}\oslash f(x+c)\bigr)+\\
& & \qquad +m\biggl(r, f(x+c) \oslash \bigl(\bigotimes_{k=1}^p
(f\oplus a_k)\bigr)\biggr),
\end{eqnarray*}
by using Jensen's formula.
\end{proof}

The following inequality is important below as a replacement to the
usual partial fraction decomposition applied in the proof of the
classical Second Main Theorem:

\begin{lemma} \label{prod-sum}
For any $p\in\mathbb{N}$, any $a_k\in\mathbb{R}$ $(1\leq k\leq p)$,
and for $b_k:=-(p-1)a_k$ $(1\leq k\leq p)$, we have
\begin{equation} \label{eqn:prod-sum}
1_{\circ}\oslash \bigl(\bigotimes_{k=1}^p (x \oplus a_k)\bigr) \leq
\bigotimes_{k=1}^p \bigl(b_k \oslash (x \oplus a_k)\bigr)\,,
\end{equation}
for any $x\in\mathbb{R}$.
\end{lemma}

\begin{proof}
Inequality (\ref{eqn:prod-sum}) is equivalent to
\begin{equation} \label{eqn:rewrite}
\sum_{k=1}^p \max(x, a_k) \geq \min_{1\leq k\leq p} \bigl(\max(x,
a_k)+(p-1)a_k\bigr)\,.
\end{equation}
Here, we may assume without loss of generality
$$
a_1\leq a_2\leq \cdots \leq a_p.
$$

{\it Case $i)$:} Suppose first that $x\leq a_{1}$. Then
$$
\max(x, a_k)=a_k \quad (1\leq k \leq p)
$$
and thus the left-hand side of (\ref{eqn:rewrite}) becomes
$$
\sum_{k=1}^p a_k \geq pa_1\,,
$$
while the right-hand side of (\ref{eqn:rewrite}) is
$$
p \min_{1\leq k\leq p} a_k=p a_1\,.
$$
Hence (\ref{eqn:rewrite}) holds in this case.

\medskip

{\it Case $ii)$:} If $a_{j-1}\leq x\leq a_{j}$ for some $2\leq j\leq
p$, then
\[
\max(x, a_k)= \left\{
\begin{array}{cl}
x & (1\leq k\leq j-1) \\
a_k \quad (j \leq k \leq p)
\end{array}
\right.\,.
\]
Thus the left-hand side of (\ref{eqn:rewrite}) becomes
$$
(j-1) x + \sum_{k=j}^p a_k \geq x+ (j-2) a_{j-1} + \sum_{k=j}^p a_k
\geq x + (p-1)a_{k-1}\,,
$$
and the right-hand side of (\ref{eqn:rewrite}) becomes
$$
\min \bigl\{ x+(p-1)a_1\,, \ \ldots \,, \ x+(p-1)a_{k-1}\,, \
pa_k\,, \ \ldots \, , \ p a_p \bigr\} \leq x+(p-1) a_{k-1}\,,
$$
verifying (\ref{eqn:rewrite}) in this case.

\medskip

{\it Case $iii)$:} If $a_{p}\leq x$, then we have
$$
\max(x, a_k)=x \quad (1\leq k\leq p)\,.
$$
The left-hand side of (\ref{eqn:rewrite}) is now
$$
p x \geq x+(p-1)a_1,
$$
while the right-hand side of (\ref{eqn:rewrite}) becomes
$$
\min_{1\leq k\leq p} \bigl(x+(p-1)a_k\bigr) \leq x+(p-1) a_1\,,
$$
proving the remaining case of (\ref{eqn:rewrite}).
\end{proof}

\begin{lemma} \label{lemma2}
For any $p\in\mathbb{N}$ and any $a_k\in\mathbb{R}$ $(1\leq k\leq
p)$ and any $c\in\mathbb{R}\setminus\{0\}$, we have
\begin{multline*}
m\biggl(r,f(x+c) \oslash \bigl(\bigotimes_{k=1}^p (f\oplus
a_{k})\bigr)\biggr) \\ \leq m\left(r, \bigoplus_{k=1}^p
                 \Bigl(f(x+c)\oslash \bigl(f(x)\oplus a_k\bigr) \Bigr) \right)
+(p-1)\max_{1\leq j\leq p} \max(a_j, 0)\,.
\end{multline*}
\end{lemma}

\begin{proof} First, applying Lemma~\ref{prod-sum}, we see
\begin{eqnarray*}
f(x+c) \oslash \biggl(\bigotimes_{k=1}^p \bigl(f(x) \oplus
a_k\bigr)\biggr) &=&
f(x+c) \otimes \Bigl( \bigoplus_{k=1}^p b_k \oslash \bigl(f(x) \oplus a_k \bigr)\Bigr) \\
&=&
\bigoplus_{k=1}^p \left\{ f(x+c) \otimes \Bigl( b_k \oslash \bigl(f(x)\oplus a_k\bigr) \Bigr) \right\} \\
&=&
\max_{1\leq k\leq p} \left\{ f(x+c) + \Bigl( b_k - \max \bigl(f(x), a_k\bigr) \Bigr) \right\} \\
&\leq &
\max_{1\leq k\leq p} \left\{ b_k + \Bigl(f(x+c)- \max\bigl(f(x), a_k\bigr)\Bigr) \right\} \\
&\leq & (p-1)\max_{1\leq k\leq p} a_k
   + \bigoplus_{k=1}^p \Bigl(f(x+c)\oslash \bigl(f(x)\oplus a_k\bigr)\Bigr)\,.
\end{eqnarray*}
By monotonicity of $m(r,\ast)$, the asserted inequality follows from
the definition of the proximity function.
\end{proof}

\begin{remark} Observe that $f(x+c)\oslash \bigl(f(x)\oplus a_k\bigr)\leq
f(x+c)\oslash f(x)$ for each~$k$.
\end{remark}

\begin{lemma} \label{lemma3}
For any $p\in\mathbb{N}$ and any $a_k\in\mathbb{R}$ $(1\leq k\leq
p)$, we have
\begin{eqnarray*}
& & \hspace{-1cm} T\biggl(r,1_{\circ}\oslash \bigl(\bigotimes_{k=1}^p (f\oplus a_k)\bigr)\biggr) \\
&=&m\biggl(r, 1_{\circ}\oslash \bigl(\otimes_{k=1}^p (f\oplus
a_k)\bigr)\biggr)+ N\biggl(r,1_{\circ}\oslash
\bigl(\bigotimes_{k=1}^p (f\oplus a_k)\bigr)\biggr) \\
&=&T\Bigl(r,\bigotimes_{k=1}^p (f\oplus a_k)\Bigr)
-\bigotimes_{k=1}^p \bigl(f(0)\oplus a_k\bigr)\,.
\end{eqnarray*}
\end{lemma}

\begin{proof} We may apply the tropical Jensen formula to the function
$F(x):=\bigotimes_{k=1}^p (f\oplus a_k)$ to obtain the above
identity. Note that
$$
F(0)=\bigotimes_{k=1}^p \bigl(f(0)\oplus a_k\bigr)=\sum_{k=1}^p \max
\bigl(f(0),  a_k\bigr).
$$
\end{proof}

\begin{lemma} \label{lemma4}
For any $p\in\mathbb{N}$ and any $a_k\in\mathbb{R}$ $(1\leq k\leq
p)$, we have
$$
N\biggl(r, 1_{\circ}\oslash \bigl(\bigotimes_{k=1}^p (f\oplus
a_k)\bigr)\biggr) \leq \sum_{k=1}^p N\bigl(r, 1_{\circ}\oslash
(f\oplus a_k)\bigr)\,.
$$
\end{lemma}

\begin{proof}
First, recall the linearity of $\omega_f(x)$ at each point $x$ with
respect to $f$, that is,
$$
\omega_{g+h}(x)=\omega_g(x)+\omega_h(x)\,.
$$
Therefore,
$$
\max\{\omega_{g+h}(x), 0\}\leq \max\{\omega_g(x),
0\}+\max\{\omega_h(x), 0\}
$$
for any $x\in \mathbb{R}$, and so $n(t, g+h) \leq n(t, g)+n(t, h)$
for any $t>0$. Hence,
$$
N(r, g\otimes h) \leq N(r, g)+N(r, h)
$$
holds for any $r>0$. The desired inequality now follows from
$$
1_{\circ}\oslash \bigl(\bigotimes_{k=1}^p (f\oplus a_k)\bigr)
=-\sum_{k=1}^p (f\oplus a_k) = \sum_{k=1}^p \bigl\{1_{\circ}\oslash
(f\oplus a_k)\bigr\} = \bigotimes_{k=1}^p \bigl\{1_{\circ}\oslash
(f\oplus a_k)\bigr\} \,.
$$
\end{proof}

\begin{lemma} \label{lemma5}
For any $p\in\mathbb{N}$ and any $a_k\in\mathbb{R}$ $(1\leq k\leq
p)$, we have
$$
T\bigl(r, \bigotimes_{k=1}^p (f\oplus a_k)\bigr)\leq p\, T(r,f)+
\sum_{k=1}^p \max(a_k, 0)\,.
$$
\end{lemma}

\begin{proof}
A straightforward reasoning by using $T(r,g\otimes h)\leq T(r,
g)+T(r, h)$ and $T(r,g\oplus  h)\leq T(r, g)+T(r, h)$ directly
confirms that
\begin{eqnarray*}
T\bigl(r, \bigotimes_{k=1}^p (f\oplus a_k)\bigr)
&\leq & \sum_{k=1}^p T(r, f\oplus a_k)\\
&\leq & \sum_{k=1}^p \bigl\{T(r, f) + T(r, a_k)\bigr\}\\
&=& p\, T(r,f) + \sum_{k=1}^p \max(a_k, 0)\,,
\end{eqnarray*}
since $T(r,a)=m(r,a)=\max(a, 0)$ for any constant $a\in \mathbb{R}$.
\end{proof}

\vspace{1ex}

In order to show a related reversed inequality to Lemma
\ref{lemma5}, we first prove the following

\begin{lemma} \label{lemma6}
For any $p\in\mathbb{N}$ and any $a_k\in\mathbb{R}$ $(1\leq k\leq
p)$, we have
$$
\max \left\{ \sum_{k=1}^p \max (f, a_k)\bigr), p\, \max (a_1, \ldots
, a_p) \right\} = p\, \max\bigl(f, \max (a_1, \ldots , a_p)
\bigr)\,,
$$
that is,
$$
\biggl(\bigotimes_{k=1}^p (f\oplus a_k)\biggr) \oplus
\biggl(\bigoplus_{k=1}^p a_k\biggr)^{\otimes p} =\biggl(f\oplus
\bigl(\bigoplus_{k=1}^p a_k\bigr)\biggr)^{\otimes p}\,.
$$
\end{lemma}

\begin{proof}
If $f\leq \max (a_1, \ldots , a_p)$, then
$$
\sum_{k=1}^p \max (f, a_k) \leq  p\, \max (a_1, \ldots , a_p),
$$
while if $f> \max (a_1, \ldots , a_p)$, then
$$
\sum_{k=1}^p \max (f, a_k) \geq p\, \max (a_1, \ldots , a_p)
$$
The assertion immediately follows.
\end{proof}

\vspace{1ex}

As an application of Lemma \ref{lemma6}, we obtain

\begin{lemma} \label{lemma7}
For any $p\in\mathbb{N}$ and any $a_k\in\mathbb{R}$ $(1\leq k\leq
p)$ with all $a_k<L_f$, we have
$$
T\bigl(r, \bigotimes_{k=1}^p (f\oplus a_k)\bigr)\geq p\,  T(r,f) -
p\, \max_{1\leq k \leq p} \max(a_k, 0) \,.
$$
\end{lemma}

\begin{proof} By Lemma~\ref{lemma6} together with $T(r,g^{\otimes p})=p\,
T(r,g)$, we have
$$
T\left(r, \bigl(\bigotimes_{k=1}^p (f\oplus a_k)\bigr) \oplus
\bigl(\bigoplus_{k=1}^p a_k\bigr)^{\otimes p}\right) =p\, T\biggl(r,
f\oplus \bigl( \bigoplus_{k=1}^p a_k \bigr)\biggr)\,.
$$
Clearly,
$$p\, T\biggl(r,f\oplus \bigl( \bigoplus_{k=1}^p a_k \bigr)\biggr)\geq pT(r,f).$$
In fact, since $f\oplus \bigl( \bigoplus_{k=1}^p
a_k\bigr)=\max\bigl\{f, \bigoplus_{k=1}^p a_k \bigr\}\geq f$, we see
that
$$
m\biggl(r, f\oplus \bigl( \bigoplus_{k=1}^p a_k \bigr)\biggr)\geq
m(r,f)\,,
$$
while
$$
N\biggl(r, f\oplus \bigl( \bigoplus_{k=1}^p a_k \bigr)\biggr)=N(r,f)
$$
holds, since $\bigoplus_{k=1}^p a_k=\max(a_1, \ldots , a_p)<L_f$.

\medskip

On the other hand, since $T(r, g\oplus h)\leq T(r,g)+T(r, h)$, we
have
$$
T\left(r, \bigl(\bigotimes_{k=1}^p (f\oplus a_k)\bigr) \oplus
\bigl(\bigoplus_{k=1}^p a_k\bigr)^{\otimes p}\right) \leq T\biggl(r,
\bigotimes_{k=1}^p (f\oplus a_k)\biggr)  + p\, T\biggl(r,
\bigoplus_{k=1}^p a_k\biggr)\,.
$$
Recalling again
$$
T\biggl(r, \bigoplus_{k=1}^p a_k\biggr)=\max\bigl((\max_{k=1}^p
a_k), 0)\bigr)=\max_{1\leq k \leq p} \max(a_k, 0),
$$
we obtain
$$
T\biggl(r, \bigotimes_{k=1}^p (f\oplus a_k)\biggr)  + p\,
\max_{1\leq k \leq p} \max(a_k, 0) \geq p\, T(r,f)\,,
$$
as desired.
\end{proof}

\begin{remark}
We now have the estimate
$$
\left| T\bigl(r, \bigotimes_{k=1}^p (f\oplus a_k)\bigr) - p\, T(r,f)
\right| \leq p\, \max_{k=1}^p \max(a_k,0)\,.
$$
under the assumption $\max(a_1, \ldots , a_p)<L_f$. Therefore,
$T\bigl(r, \bigotimes_{k=1}^p (f\oplus a_k)\bigr) = p\, T(r,f)$
whenever $a_{k}\leq 0$ for each $k=1,\ldots ,p$. This may seem a bit
curious. Recall, however, the assumption $a_k<L_f$ and its strong
consequence $N(r, f\oplus a_k)=N(r,f)$ for each $k$.
\end{remark}

{\it Proof of Theorem~\ref{thm:second}.} It follows by combining
Lemmas~\ref{lemma1}, \ref{lemma2} and \ref{lemma3} - \ref{lemma7}
above that
\begin{multline*}
T\bigl(r, f(x+c)\bigr) - N\bigl(r, 1_{\circ}\oslash f(x+c)\bigr)
+ m\bigl(r,f(x+c)\oslash f(x)\bigr) - f(c) \\
\geq \ pT(r,f) - \sum_{k=1}^p N\bigl(r, 1_{\circ}\oslash (f\oplus a_k)\bigr) \\
\qquad -(2p-1)\max_{1\leq k\leq p} \max(a_k, 0) - \sum_{k=1}^p \max
\bigl(f(0),  a_k\bigr)\,,
\end{multline*}
and therefore
\begin{multline} \label{eqn:SMT}
pT(r,f) \leq
\sum_{k=1}^p N\bigl(r, 1_{\circ}\oslash (f\oplus a_k)\bigr) + T\bigl(r, f(x+c)\bigr) \\
\qquad - N\bigl(r, 1_{\circ}\oslash f(x+c)\bigr)
+ m\bigl(r,f(x+c)\oslash f(x)\bigr) - f(c) \\
+(2p-1)\max_{1\leq k\leq p} \max(a_k, 0) + \sum_{k=1}^p \max
\bigl(f(0), a_k\bigr)\,,
\end{multline}
completing the proof.

\section{Valiron--Mohon'ko, Mohon'ko and Clunie lemmas}\label{VMC}

In this section we prove slightly extended versions of three results
from the restricted tropical setting of integer slopes, see
\cite{LY2}. These results are counterparts, in some sense, of three
classical lemmas, frequently used in applications of Nevanlinna
theory. As for the classical background, we first recall a lemma due
to Valiron and Mohon'ko, see e.g. \cite{L}, Theorem 2.2.5, and
another one due to Mohon'ko, see e.g. \cite{L}, Proposition 9.2.3.
In the restricted tropical setting of integer slopes, we refer to
\cite{LY2}.

\medskip

Before proceeding to formulate these results in a slightly extended
setting, we need to define a tropical difference {\it Laurent}
polynomials in a tropical function and its shifts. This notion is a
slight generalization of a difference polynomial considered in
\cite{L} and \cite{LY2} in the sense that exponents can be negative
and real. Let $\lambda =(\lambda_{0}, \lambda_{1}, \ldots
,\lambda_{m})$ be a multi-index of real numbers, and consider
\begin{eqnarray*}
f(x \uplus c)^{\otimes \lambda}
&:=&\bigotimes_{j=0}^{m} f(x+c_{j})^{\otimes \lambda_{j}}\\
&=&\lambda_{0}f(x)+\lambda_{1}f(x+c_{1})+\cdots
+\lambda_{m}f(x+c_{m})
\end{eqnarray*}
with given shifts $c_{1},\ldots ,c_{m}$ in $[0,+\infty )$. Also we
will write
$$
\bigl\{ f(x \uplus c) \oslash f(x) \bigr\}^{\otimes \lambda}
:=\bigotimes_{j=1}^{m} \bigl\{f(x+c_{j})-f(x)\bigr\}^{\otimes
\lambda_{j}}
$$

Then, an expression of the form
$$
P(x,f)
=\bigoplus_{\lambda\in\Lambda}^{} a_{\lambda}(x)\otimes f(x\uplus c)^{\otimes \lambda} \\
=\max_{\lambda\in\Lambda} \Bigl\{ a_{\lambda}(x) + \sum_{j=0}^{m}
\lambda_j f(x+c_j) \Bigr\}
$$
with tropical meromorphic coefficients $a_{\lambda}(x)$ $(\lambda\in
\Lambda)$ over a finite set $\Lambda=\Lambda[P]$ of real indices, is
called a tropical difference Laurent polynomial of total degree
$$
\deg (P):=\max_{\lambda\in\Lambda[P]} \|\lambda\| \ (\in \mathbb{R})
$$
in $f$ and its shifts, with $\|\lambda\|:=\lambda_{0}+\cdots
+\lambda_{m}$. We also denote
$$
\Omega[P](x):=\bigoplus_{\lambda\in \Lambda[P]} a_{\lambda}(x)
=\max_{\lambda\in \Lambda[P]} a_{\lambda}(x),
$$
and
$$
\overline{\Omega}[P](x):=\bigoplus_{\lambda\in \Lambda[P]}
\bigl(-a_{\lambda}(x)\bigr) =\max_{\lambda\in \Lambda[P]}
\bigl(-a_{\lambda}(x)\bigr).
$$
Note that $\overline{\Omega}[P](x)$ does not coincide with
$\Omega[1_{\circ}\oslash P](x)$, since
\begin{eqnarray*}
1_{\circ}\oslash P(x,f)
&=&-\bigoplus_{\lambda\in\Lambda}^{} a_{\lambda}(x)\otimes f(x\uplus c)^{\otimes \lambda}\\
&=&\min_{\lambda\in\Lambda}\Bigl(-a_{\lambda}(x) - \sum_{j=0}^{m}
\lambda_j f(x+c_j) \Bigr).
\end{eqnarray*}
In what follows, we also have a need to consider leading
coefficients in in a difference Laurent polynomial. To this end, we
denote
$$
\Upsilon [P](x) :=\bigoplus_{\lambda\in \hat{\Lambda}[P]}
a_{\lambda}(x) =\max_{\lambda\in \hat{\Lambda}[P]} a_{\lambda}(x),
$$
where $\hat{\Lambda}[P]:=\{ \lambda\in\Lambda[P] \, : \,
\|\lambda\|=\deg(P)\}$. The notation $\bigoplus$ is used to
emphasizing that the sum in question stands for a tropical sum.

\medskip

In this section, we complement previous results by assuming that $f$
is a tropical meromorphic function of hyper-order $\rho_{2}<1$.
Recalling Lemma \ref{technical} and Proposition \ref{LLDh}, we say,
in the case of this special situation that the coefficients
$a_{\lambda}$ of a tropical difference Laurent polynomial are small
(in the tropical sense) with respect to~$f$, if
$T(r,a_{\lambda})=S_{\delta}(r,f)$ holds with a quantity
$S_{\delta}(r,f)=o\bigl(T(r,f)/r^{\delta}\bigr)$ outside of a set of
finite logarithmic measure, where $0<\delta <1-\rho_{2}$.

\medskip

The proofs in this section rely on the notion of the proximity
function only, in addition to completely elementary analysis.
Therefore, the proofs in~\cite{LY2} basically carry over to the
present situation, with natural modifications. However, due to the
change from difference polynomials to difference Laurent
polynomials, we repeat the key points of the proofs here, for the
convenience of the reader.

\begin{proposition}\label{Ptof} Suppose that $f$ is tropical meromorphic and
a tropical difference Laurent polynomial $P(x,f)$ has a term
$a_{\lambda}(x)\otimes f(x\uplus c)^{\otimes \lambda}$ with $\|
\lambda\|>0$. Then, putting $\lambda_j^+=\max(\lambda_j,0)$ and
$\lambda_j^-=\max(-\lambda_j,0)$, we have
\begin{eqnarray*}
& & \hspace{-7ex} \| \lambda \| m(r,f) \\
&\leq & \sum_{j=1}^m \left\{ \lambda_j^+ m\bigl(r, f(x)\oslash
f(x+c_j)\bigr)
             + \lambda_j^- m\bigl(r, f(x+c_j)\oslash f(x)\bigr) \right\} \\
& & \quad  +m(r, 1_{\circ}\oslash a_{\lambda}) + m\bigl(r,
P(x,f)\bigr).
\end{eqnarray*}
In particular, if $f$ is of hyper-order $\rho_{2}<1$ and both
$m(r,1_{\circ}\oslash a_{\lambda})$ and $m(r,P(x,f))$ are of
$S_{\delta}(r,f)$, then $m(r,f)=S_{\delta}(r,f)$.
\end{proposition}

\begin{proof} First we have
\begin{eqnarray*}
& & \hspace{-7ex} \|\lambda\|m(r,f) \\
&=&m(r, \|\lambda\| f) \\
&=&m\Bigl(r, \sum_{j=0}^m \lambda_j \bigl(f(x)-f(x+c_j)\bigr)
        +\sum_{j=0}^m \lambda_j f(x+c_j)+a_{\lambda}(x)-a_{\lambda}(x)\Bigr)\\
&\leq& m\Bigl(r, \sum_{j=0}^m \lambda_j
\bigl(f(x)-f(x+c_j)\bigr)\Bigr)
         + m\Bigl(r, a_{\lambda}(x)+\sum_{j=0}^m \lambda_j f(x+c_j)\Bigr) \\
& & \quad +m(r, -a_{\lambda}).
\end{eqnarray*}
Since $\lambda_j=\lambda_j^+ - \lambda_j^-$ $(0\leq j\leq m)$, the
sum $\sum_{j=0}^m \lambda_j \bigl(f(x)-f(x+c_j)\bigr)$ can be
written as
$$
\sum_{j=0}^m \Bigl\{ \lambda_j^+ \bigl(f(x)\oslash f(x+c_j)\bigr)
 + \lambda_j^- \bigl(f(x+c_j)\oslash f(x)\bigr) \Bigr\}.
$$
Further, $a_{\lambda}(x)+\sum_{j=0}^m \lambda_j f(x+c_j) \leq
P(x,f)$ by the definition of the tropical difference Laurent
polynomial
$$
P(x,f) =\max_{\lambda\in\Lambda}
  \Bigl\{ a_{\lambda}(x) + \sum_{j=1}^m \lambda_j f(x+c_j) \Bigr\}.
$$
The desired inequality now immediately follows from the previous
observations. The special case of hyper-order $< 1$ is an immediate
consequence, recalling only Proposition \ref{LLDh}.
\end{proof}

The following theorem may be understood as a partial tropical
counterpart of the classical Valiron--Mohon'ko lemma, see
\cite{LY2}, Theorem 2.3:

\begin{theorem}\label{polyVM} Given a tropical meromorphic function $f$ and
its tropical difference Laurent polynomial
$$
P(x,f) =\bigoplus_{\lambda\in\Lambda}
  \Bigl\{ a_{\lambda}(x)\otimes f(x\uplus c)^{\otimes\lambda} \Bigr\}.
$$
Then
\begin{eqnarray*}
& & \hspace{-5ex} \left|m\bigl(r, P(x,f)\bigr) - m \bigl(r, \deg(P)f(x)\bigr)\right| \\
&\leq& \max\left\{m\bigl(r, \Omega[P](x)\bigr) +m\Bigl(r,
\max_{\lambda\in\Lambda}
       \bigl(f(x\uplus c)\oslash f(x) \bigr)^{\otimes  \lambda} \Bigr), \right.\\
& &  \qquad
    \left. m\bigl(r, \overline{\Omega}[P](x)\bigr)
 +m\Bigl(r, \max_{\lambda\in\Lambda}
       \bigl(f(x\uplus c)\oslash f(x) \bigr)^{\otimes  (-\lambda)} \Bigr)\right\}
\end{eqnarray*}

In particular, if $f$ is of hyper-order $\rho_{2}<1$,  $\deg (P)>0$
and the coefficients of $P(x,f)$ are all small with respect to $f$,
then
$$
m\bigl(r,P(z,f)\bigr)=\deg(P) m(r,f)+S_{\delta}(r,f).
$$
\end{theorem}

\begin{proof} Recalling
$$
P(x,f) \leq \max_{\lambda\in\Lambda} a_{\lambda}(x) +
\max_{\lambda\in\Lambda} \sum_{j=0}^{m} \lambda_j
\bigl\{f(x+c_j)-f(x)\bigr\} + \deg(P) f(x),
$$
we have
\begin{eqnarray*}
m\bigl(r, P(x,f)\bigr) &\leq & m\bigl(r, \Omega[P](x)\big)
+m\Bigl(r, \max_{\lambda\in\Lambda}
       \bigl(f(x\uplus c)\oslash f(x)\bigr)^{\otimes \lambda} \Bigr)\\
& & \quad m\bigl(r, \deg(P)f(x)\bigr)\,.
\end{eqnarray*}
On the other hand, for any $\lambda\in\Lambda$,
$$
P(x,f) \geq a_{\lambda}(x)
   + \sum_{j=0}^{m} \lambda_j \bigl\{f(x+c_j)-f(x)\bigr\} + \|\lambda\|f(x),
$$
that is,
$$
\|\lambda\|f(x)\leq P(x,f) - a_{\lambda}(x) - \sum_{j=0}^{m}
\lambda_j \bigl\{f(x+c_j)-f(x)\bigr\}
$$
holds, and therefore
$$
\deg(P)f(x)\leq P(x,f) +\overline{\Omega}[P](x) +
\max_{\lambda\in\Lambda} \sum_{j=0}^{m} (-\lambda_j)
\bigl\{f(x+c_j)-f(x)\bigr\}.
$$
This implies
\begin{eqnarray*}
m\bigl(r, \deg(P)f(x)\bigr) &\leq & m\bigl(r, P(x,f)\bigr) +
m\bigl(r, \overline{\Omega}[P](x)\bigr)\\
& & +m\Bigl(r, \max_{\lambda\in\Lambda}
   \bigl(f(x\uplus c)\oslash f(x)\bigr)^{\otimes (-\lambda)} \Bigr)\,,
\end{eqnarray*}
completing the proof.
\end{proof}

As a special case of Theorem~\ref{polyVM}, we have

\begin{corollary}\label{poly} Given a tropical
meromorphic function of hyper-order $\rho_{2}<1$ and its tropical
polynomial of degree $n$,
$$
P(x,f)=\bigoplus_{j=0}^{n}a_{j}(x)\otimes f(x+c_j)^{\otimes j},
$$
then
$$
m\bigl(r,P(z,f)\bigr)=nm(r,f)+S_{\delta}(r,f).
$$
\end{corollary}

\bigskip

The next theorem is related, in the spirit, to the Mohon'ko lemma.
However, it cannot be considered as a complete tropical counterpart
to that. Indeed, the assumptions below mean that $P(x,0)=0$, while
one always has $P(x,0)\neq 0$ in the classical setting. As for the
restricted case of integer slopes, see \cite{LY2}, Theorem 2.5:

\begin{theorem}\label{Moh} Let $f$ be a tropical meromorphic solution of
a tropical difference Laurent polynomial equation
$$
P(x,f) =\bigoplus_{\lambda\in\Lambda}
   a_{\lambda}(x)\otimes f(x\uplus c)^{\otimes \lambda}
=0
$$
such that $\|\lambda\|\neq 0$ for all $\lambda\in\Lambda$. Then any
of $m(r,f)$ and $m\bigl(r, 1_{\circ}\oslash (f\oplus a)\bigr)$
$(a\in\mathbb{R})$ is not greater than
\begin{multline*}
\left(\max_{\lambda\in\Lambda}\left|\frac{1}{\|\lambda\|}\right|\right)
\Bigl\{m\bigl(r, \Omega[P](x)\bigr)+m\bigl(r, \overline{\Omega}[P](x)\bigr)\Bigr\} \\
+ \sum_{j=0}^{m} \max_{\lambda\in\Lambda}
\left|\frac{\lambda_{j}}{\|\lambda\|}\right|
       \Bigl\{ m\bigl(r, f(x+c_j)\oslash f(x)\bigr)
        + m\bigl(r, f(x)\oslash f(x+c_j)\bigr) \Bigr\}\,.
\end{multline*}

In particular, if $f$ is of hyper-order $\rho_{2}<1$ and a solution
of a tropical difference polynomial equation with small coefficients
$$
P(x,f) =\bigoplus_{\lambda\in\Lambda}a_{\lambda}(x)\otimes
f^{\otimes \lambda}(x) =0
$$
such that $\|\lambda\|\geq 1$ for all $\lambda\in\Lambda$. Then
$$
m(r,f)=S_{\delta}(r,f)), \qquad m(r,-f)=S_{\delta}(r,f)
$$
and for any $a\in\mathbb{R}\setminus\{0\}$
$$
m\bigl(r,1_{\circ}\oslash (f\oplus a)\bigr)
  =m(r,-\max (f,a))=S_{\delta}(r,f).
$$
\end{theorem}

\begin{proof} It follows from $P(x,f)=0$ that
$$
a_{\lambda}(x)+\sum_{j=0}^{m} \lambda_j f(x+c_j)\leq 0 \quad
(x\in\mathbb{R})
$$
for any $\lambda\in\Lambda$, while for each $x\in\mathbb{R}$ there
exists a $\lambda_x=(\lambda_{x,1} , \ldots , \lambda_{x,m})\in
\Lambda$ such that the equality, or
$$
a_{\lambda_{x}}(x)+\sum_{j=0}^{m} \lambda_{x,j}
\big\{f(x+c_j)-f(x)\bigr\} =-\|\lambda_{x}\|f(x)
$$
holds. Since $\|\lambda_x\|\neq 0$ by assumption, we have
\begin{eqnarray*}
f(x)&=& -\frac{1}{\|\lambda_{x}\|}a_{\lambda_{x}}(x)
        -\sum_{j=0}^{m}\frac{\lambda_{x,j}}{\|\lambda_{x}\|}
                        \bigl\{f(x+c_j)-f(x)\bigr\}
         \quad \text{and}\\
f(-x)&=& -\frac{1}{\|\lambda_{-x}\|}a_{\lambda_{-x}}(x)
        -\sum_{j=0}^{m}\frac{\lambda_{-x,j}}{\|\lambda_{-x}\|}
                        \bigl\{f(-x+c_j)-f(-x)\bigr\}\,,
\end{eqnarray*}
so that
\begin{eqnarray*}
& & \hspace{-5ex} f(x)^{+} \\
&\leq & \left|\frac{1}{\|\lambda_{x}\|}\right|
      \Bigl\{ \bigl(a_{\lambda_{x}}(x)\bigr)^{+}
            + \bigl(-a_{\lambda_{x}}(x)\bigr)^{+} \Bigr\} \\
& &  \hspace{-5ex}
        + \sum_{j=0}^{m}\left|\frac{\lambda_{x,j}}{\|\lambda_{x}\|}\right|
       \Biggl\{ \bigl\{f(x+c_j)-f(x)\bigr\}^{+}
              + \bigl\{f(x)-f(x+c_j)\bigr\}^{+} \Biggr\}\,,\\
&\text{and}& \\
& & \hspace{-5ex} f(-x)^{+} \\
&\leq & \left|\frac{1}{\|\lambda_{-x}\|}\right|
      \Biggl\{ \bigl(a_{\lambda_{-x}}(-x)\bigr)^{+}
             + \bigl(-a_{\lambda_{-x}}(-x)\bigr)^{+} \Biggr\} \\
& & \hspace{-8ex}
       + \sum_{j=0}^{m}\left|\frac{\lambda_{-x,j}}{\|\lambda_{-x}\|}\right|
       \Biggl\{ \bigl\{f(-x+c_j)-f(-x)\bigr\}^{+}
              + \bigl\{f(-x)-f(-x+c_j)\bigr\}^{+} \Biggr\}\,,
\end{eqnarray*}
for each $x\in\mathbb{R}$. Thus we obtain further $f(x)^{+}$ is not
greater than
\begin{multline*}
\max_{\lambda\in\Lambda} \left|\frac{1}{\|\lambda\|}\right|
\Bigl\{ \bigl(\Omega[P](x)\bigr)^{+}
              + \bigl(\overline{\Omega}[P](x)\bigr)^{+} \Bigr\} \\
+ \sum_{j=0}^{m} \max_{\lambda\in\Lambda}
                 \left|\frac{\lambda_{j}}{\|\lambda\|}\right|
        \Biggl\{ \bigl\{f(x+c_j)-f(x)\bigr\}^{+}
               + \bigl\{f(x)-f(x+c_j)\bigr\}^{+} \Biggr\}\,,
\end{multline*}
and $f(-x)^{+}$ is not greater than
\begin{multline*}
\max_{\lambda\in\Lambda} \left|\frac{1}{\|\lambda\|}\right|
      \Bigl\{ \bigl(\Omega[P](-x)\bigr)^{+}
            + \bigl(\overline{\Omega}[P](-x)\bigr)^{+} \Bigr\} \\
+ \sum_{j=0}^{m} \max_{\lambda\in\Lambda}
             \left|\frac{\lambda_{j}}{\|\lambda\|}\right|
       \Biggl\{ \bigl\{f(-x+c_j)-f(-x)\bigr\}^{+}
              + \bigl\{f(-x)-f(-x+c_j)\bigr\}^{+} \Biggr\}\,.
\end{multline*}
Now, by definition, the first result
\begin{eqnarray*}
& & \hspace{-5ex} m(r,f)
= \frac{1}{2}\big\{f(x)^{+}+f(-r)\bigr\}\\
&\leq &  \max_{\lambda\in\Lambda}
            \left|\frac{1}{\|\lambda\|}\right|
             \Biggl\{ m(r,\Omega[P](x)\bigr)
                    + m(r,\overline{\Omega}[P](x)\bigr) \Biggr\}\\
&   &  + \sum_{j=0}^{m} \max_{\lambda\in\Lambda}
            \left|\frac{\lambda_{j}}{\|\lambda\|}\right|
             \Biggl\{ m\bigl(r, f(x+c_j)\oslash f(x)\bigr)
                    + m\bigl(r, f(x)\oslash f(x+c_j)\bigr) \Biggr\}
\end{eqnarray*}
is obtained with the two estimates above.

Similarly,
\begin{eqnarray*}
-f(x)&=& \frac{1}{\|\lambda_{x}\|}a_{\lambda_{x}}(x)
        +\sum_{j=0}^{m}\frac{\lambda_{x,j}}{\|\lambda_{x}\|}
                         \bigl\{f(x+c_j)-f(x)\bigr\}
         \quad \text{and}\\
-f(-x)&=& \frac{1}{\|\lambda_{-x}\|}a_{\lambda_{-x}}(x)
        +\sum_{j=0}^{m}\frac{\lambda_{-x,j}}{\|\lambda_{-x}\|}
                         \bigl\{f(-x+c_j)-f(-x)\bigr\}\,,
\end{eqnarray*}
and thus $\bigl(-f(x)\bigr)^{+}$ is not greater than
\begin{multline*}
\max_{\lambda\in\Lambda} \left|\frac{1}{\|\lambda\|}\right|
        \Bigl\{ \bigl(\Omega[P](x)\bigr)^{+}
              + \bigl(\overline{\Omega}[P](x)\bigr)^{+} \Bigr\} \\
 + \sum_{j=0}^{m} \max_{\lambda\in\Lambda}
              \left|\frac{\lambda_{j}}{\|\lambda\|}\right|
                \Bigl\{ \bigl\{f(x+c_j)-f(x)\bigr\}^{+}
                      + \bigl\{f(x)-f(x+c_j)\bigr\}^{+} \Bigr\}\,,
\end{multline*}
and $\bigl(-f(-x)\bigr)^{+}$ is not greater than
\begin{multline*}
\max_{\lambda\in\Lambda} \left|\frac{1}{\|\lambda\|}\right|
      \Bigl\{ \bigl(\Omega[P](-x)\bigr)^{+}
            + \bigl(\overline{\Omega}[P](-x)\bigr)^{+} \Bigr\} \\
+ \sum_{j=0}^{m} \max_{\lambda\in\Lambda}
            \left|\frac{\lambda_{j}}{\|\lambda\|}\right|
                 \Bigl\{ \bigl\{f(-x+c_j)-f(-x)\bigr\}^{+}
                       + \bigl\{f(-x)-f(-x+c_j)\bigr\}^{+} \Bigr\}\,.
\end{multline*}
Hence the growth of $m(r, 1_{\circ}\oslash f)$ is also bounded by
the same quantity as the above.

\medskip

Finally, since $-\max\bigl(f(x),a \bigr)\leq -f(x)$ on $\mathbb{R}$
for any $a\in\mathbb{R}$,
$$
m\bigl(r, 1_{\circ}\oslash (f\oplus a)\bigl)
  \leq m(r, 1_{\circ}\oslash f)\,,
$$
and we are done.
\end{proof}

\bigskip

To close this section, we recall the Clunie lemma. In addition to
the classical differential version, see e.g. \cite{L}, we recall the
corresponding difference version, namely the classical form of the
Clunie lemma in the case of integer slopes, see \cite{HS},
Theorem~4.5:

\begin{theorem}\label{TropC} Let $P(x,f),Q(x,f)$ be two tropical
difference polynomials with small coefficients. If $f$ is a tropical
meromorphic function satisfying equation
$$
f(x)^{\otimes n} \otimes P(x,f)=Q(x,f)
$$
such that the degree of $Q$ in $f$ and its shifts is at most $n$,
then for any $\varepsilon >0$,
$$
m\bigl(r,P(x,f)\bigr)
=O\left\{r^{-1}\Bigl(T(r+|c|,f)^{1+\varepsilon}
                       +o\bigl\{T(r+|c|,f)\bigr\}\Bigr)\right\},
$$
holds outside of an exceptional set of finite logarithmic measure.
\end{theorem}



More general versions of the Clunie lemma have been proved
in~\cite{YY} for differential polynomials and in~\cite{LY} for
difference polynomials. The following theorem is the tropical
counterpart of these versions of the Clunie lemma, see \cite{LY2}
for the same result in the case of integer slopes.

\begin{theorem}\label{tropClun} Let $H(x,f),P(x,f),Q(x,f)$ be tropical
difference Laurent polynomials in $f$ and its shifts. If $f$ is a
tropical meromorphic function satisfying equation
\begin{equation}\label{tropeq}
 H(x,f)\otimes P(x,f)=Q(x,f)
\end{equation}
such that $\deg(P)\geq 0$ and $\deg(Q)\leq \deg(H)$ in $f$ and its
shifts, then
\begin{multline}
m\bigl(r,P(x,f)\bigr) \\
\leq m\bigl(r, \Omega[P](x)\bigr)
      +m\bigl(r, \Omega[Q](x)\bigr)
        +m(r, 1_{\circ}\oslash \Upsilon [H](x)\bigr) \\
 +m\biggl(r, \max_{\lambda\in \Lambda[P]}
                  \bigl(f(x\uplus c)\oslash f(x)\bigr)^{\otimes \lambda} \biggr) \\
 +m\biggl(r, \max_{\mu \in \Lambda[Q]}
                  \bigl(f(x\uplus c)\oslash f(x)\bigr)^{\otimes \mu} \biggr)\\
 +m\biggl(r, \max_{\nu \in \hat{\Lambda}[H]}
                  \bigl(f(x\uplus c)\oslash f(x)\bigr)^{\otimes (-\nu)} \biggr).
   \label{tropest}
\end{multline}

In particular, if each of those polynomials $H(x,f),P(x,f), Q(x,f)$
has the small coefficients, then, for given $\varepsilon>0$,
\begin{equation}\label{tropineq}
 m\bigl(r,P(x,f)\bigr)
 =O\left\{r^{-1}\Bigl(T(r+|c|,f)^{1+\varepsilon}
                 +o\bigl\{T(r+|c|,f)\bigr\}\Bigr)\right\}
\end{equation}
holds outside an exceptional set of finite logarithmic measure.
\end{theorem}

\begin{proof}
Given a fixed $r>0$, we put
$$
S_{+}:=\{ s\, :\, f(s)\geq 0, |s|=r\} \quad \text{and} \quad
S_{-}:=\{ s\, :\, f(s)< 0, |s|=r\}\,,
$$
so that $S_{+}\cup S_{-}=\{\pm r\}$. Then
\begin{eqnarray} 
m\bigl(r, P(x,f)\bigr)
&=& \frac{1}{2}\bigl(P(r,f)^{+}+P(-r,f)^{+}\bigr)\\
&=& \frac{1}{2}\left(\sum_{s\in S_{+}} P(s,f)^{+}
                    +\sum_{s\in S_{-}} P(s,f)^{+}\right)\,.
\end{eqnarray}
Now we denote
\begin{eqnarray*}
P(x,f)&=&\bigoplus_{\lambda\in\Lambda[P]}^{}
            a_{\lambda}(x)\otimes f(x\uplus c)^{\otimes \lambda}\,, \\
Q(x,f)&=&\bigoplus_{\mu \in\Lambda[Q]}^{}
            b_{\mu}(x)\otimes f(x\uplus c)^{\otimes \mu}\,, \\
H(x,f)&=&\bigoplus_{\nu \in\Lambda[H]}^{}
            d_{\nu}(x)\otimes f(x\uplus c)^{\otimes \nu}\,. \\
\end{eqnarray*}

When $x \in S_{-}$, we have
\begin{eqnarray*}
P(x,f) &=&\max_{\lambda\in\Lambda[P]}
      \Bigl\{ a_{\lambda}(x)
               + \sum_{j=0}^{m} \lambda_j \bigl(f(x+c_j)-f(x)\bigr)
                 +\|\lambda\| f(x) \Bigr\} \\
&\leq & \max_{\lambda\in\Lambda[P]}
       \Bigl\{ a_{\lambda}(x)
               + \sum_{j=0}^{m} \lambda_j \bigl(f(x+c_j)-f(x)\bigr)\Bigr\}
                 +\deg(P) f(x) \\
&\leq & \Omega[P](x)
        + \max_{\lambda\in\Lambda[P]}
            \bigl(f(x\uplus c)\oslash f(x)\bigr)^{\otimes \lambda}
                         +\deg(P) f(x) \\
&\leq & \Omega[P](x)
        + \max_{\lambda\in\Lambda[P]}
            \bigl(f(x\uplus c)\oslash f(x)\bigr)^{\otimes \lambda}\,,
\end{eqnarray*}
where the last inequality follows from the present assumptions,
$\deg(P)\geq 0$ and $f(x)<0$. Thus for $s \in S_{-}$,
\begin{eqnarray}
P(s,f)^{+} &\leq & \Omega[P](s)^{+}
           + \left(\max_{\lambda\in\Lambda[P]}
                  \bigl(f(x\uplus c)\oslash f(x)\bigr)^{\otimes \lambda}\right)^{+}
            \nonumber \\
&\leq & 2 m\bigl(r, \Omega[P](s)\bigr) \nonumber \\
&     &  \qquad + 2 m \Bigl(r, \max_{\lambda\in\Lambda[P]}
                  \bigl(f(x\uplus c)\oslash f(x)\bigr)^{\otimes \lambda}\Bigr)\,.
            \label{minus}
\end{eqnarray}

When $s\in S_{+}$, we have
\begin{equation}
Q(x,f) \leq \Omega[Q](x)
        + \max_{\mu\in\Lambda[Q]} \bigl(f(x\uplus c)\oslash f(x)\bigr)^{\otimes \mu}
             +\deg(Q) f(x)\,,  \label{ineq:Q}
\end{equation}
while
\begin{eqnarray*}
& & \hspace{-7ex} Q(x,f)=H(x,f)\otimes P(x,f)\\
&=& P(x, f)
    + \max_{\nu \in\Lambda[H]}
      \Bigl\{ d_{\nu}(x)
           + \sum_{j=0}^{m} \nu_j \bigl(f(x+c_j)-f(x)\bigr)+\|\nu\| f(x) \Bigr\} \\
&\geq & P(x,f)+ \max_{\nu \in\hat{\Lambda}[H]}
      \Bigl\{ d_{\nu}(x) + \sum_{j=0}^{m} \nu_j \bigl(f(x+c_j)-f(x)\bigr)
              +\deg(H) f(x) \Bigr\} \\
&= & P(x,f)+ \max_{\nu \in\hat{\Lambda}[H]}
      \Bigl\{ d_{\nu}(x) + \sum_{j=0}^{m} \nu_j \bigl(f(x+c_j)-f(x)\bigr)\Bigr\}
              +\deg(H) f(x)\,.
\end{eqnarray*}
Further, the latter implies
\begin{eqnarray*}
& & \hspace{-7ex} P(x,f) + \deg(H) f(x) \\
&\leq & Q(x, f)
    - \max_{\nu \in\hat{\Lambda}[H]}
          \Bigl\{ d_{\nu}(x)+\sum_{j=0}^{m} \nu_j \bigl(f(x+c_j)-f(x)\bigr)\Bigr\}\\
&\leq & Q(x,f)+ \min_{\nu \in\hat{\Lambda}[H]}
     \Bigl\{ -d_{\nu}(x) + \sum_{j=0}^{m} (-\nu_j) \bigl(f(x+c_j)-f(x)\bigr)\Bigr\}\\
&\leq & Q(x,f)+ \min_{\nu \in\hat{\Lambda}[H]}\Bigl\{
-d_{\nu}(x)\Bigr\}
        + \max_{\nu \in\hat{\Lambda}[H]}
            \Bigl\{\sum_{j=0}^{m} (-\nu_j) \bigl(f(x+c_j)-f(x)\bigr)\Bigr\}\\
&\leq & Q(x,f)- \Upsilon [H](x)
         + \max_{\nu \in\hat{\Lambda}[H]}
            \bigl(f(x\uplus c)\oslash f(x)\bigr)^{\otimes (-\nu)}\,.
\end{eqnarray*}
This together with (\ref{ineq:Q}) shows
\begin{eqnarray*}
P(x,f) &\leq & \Omega[Q](x) +1_{\circ}\oslash \Upsilon [H](x)
         + \max_{\mu \in \Lambda[Q]}
            \bigl(f(x\uplus c)\oslash f(x)\bigr)^{\otimes \mu} \\
&     & + \max_{\nu \in \hat{\Lambda}[H]}
            \bigl(f(x\uplus c)\oslash f(x)\bigr)^{\otimes (-\nu)}
           +\bigl\{\deg(Q) - \deg(H)\bigr\} f(x)\\
&\leq & \Omega[Q](x) +1_{\circ}\oslash \Upsilon [H](x)
         + \max_{\mu \in \Lambda[Q]}
            \bigl(f(x\uplus c)\oslash f(x)\bigr)^{\otimes \mu} \\
&     &  + \max_{\nu \in \hat{\Lambda}[H]}
            \bigl(f(x\uplus c)\oslash f(x)\bigr)^{\otimes (-\nu)}\,,
\end{eqnarray*}
since $\bigl\{\deg(Q) - \deg(H)\bigr\} f(x)\leq 0$. Similar to the
above, we obtain
\begin{eqnarray}
P(s,f)^{+} &\leq & 2\Biggl\{ m\bigl(r, \Omega[Q](s)\bigr)
                + m\bigl(r, 1_{\circ}\oslash \Upsilon [H](x)\bigr) \\ \nonumber
&     & \qquad
            m \Bigl(r, \max_{\mu \in\Lambda[Q]}
                         \bigl(f(x\uplus c)\oslash f(x)\bigr)^{\otimes \mu}\Bigr)\\
                   \nonumber
&     & \qquad \qquad
            m \Bigl(r, \max_{\nu \in\hat{\Lambda}[H]}
                         \bigl(f(x\uplus c)\oslash f(x)\bigr)^{\otimes (-\nu)}\Bigr)\Biggr\}\,.
               \label{plus}
\end{eqnarray}
Inserting both (\ref{minus}) and (\ref{plus}) into (\ref{prox}), we
obtain the desired estimate (\ref{tropest}).
\end{proof}

\begin{corollary}\label{tropClunf} Let $H(x,f),P(x,f),Q(x,f)$ be
tropical difference polynomials in $f$ and its shifts with small
coefficients. If $f$ is a tropical meromorphic function of
hyper-order $\rho_{2}<1$ satisfying equation
\begin{equation}\label{tropeq2}
 H(x,f)\otimes P(x,f)=Q(x,f)
\end{equation}
such that $\deg (Q)\leq \deg (H)$ in $f$ and its shifts, then for
$\delta$ with $0<\delta <1-\rho_{2}$,
$$
m(r,f)=S_{\delta}(r,f)
$$
outside an exceptional set of finite logarithmic measure.
\end{corollary}

\begin{corollary}\label{tropClunfsp} Let $P(x,f),Q(x,f)$ be
tropical difference polynomials in $f$ and its shifts with small
coefficients and $\alpha\in\mathbb{R}$. If $f$ is a tropical
meromorphic function of hyper-order $\rho_{2}<1$ satisfying equation
\begin{equation}\label{tropeq3}
 f(x)^{\otimes\alpha}\otimes P(x,f)=Q(x,f)
\end{equation}
such that $\deg (Q)\leq\alpha$ in $f$ and its shifts, then for
$\delta$ with $0<\delta <1-\rho_{2}$,
$$
m(r,f)=S_{\delta}(r,f)
$$
outside an exceptional set of finite logarithmic measure.
\end{corollary}

\section{Periodic functions}\label{periodic}

In this section, we proceed to showing that non-constant tropical
periodic functions are of finite order, and more precisely, of order
two.

\medskip

Before going into details, let us consider the elementary
ultra-discrete equation
\begin{equation} \label{eqn:1}
f(x+1)=c\otimes f(x) \quad \text{i.e.} \quad f(x+1)\oslash f(x)=c,
\end{equation}
for $c\in\mathbb{R}\setminus\{0\}$. This equation has a special
solution $g_0(x)=cx=x^{\otimes c}$. Thus for an arbitrary tropical
meromorphic solution $g(x)$ to Equation~(\ref{eqn:1}), we see that
the difference $g(x)-x^{\otimes c}$ must be a tropical meromorphic
periodic function of period~$1$. Hence, an arbitrary tropical
meromorphic solution of (\ref{eqn:1}) is uniquely obtained as the
sum of the linear function $cx=x^{\otimes c}$ and a tropical
meromorphic $1$-periodic function. This sum is of order two, since
$g_0(x)$ is of order $1$ as one may immediately see:
\begin{proposition} \label{prop:0}
For any $c\in\mathbb{R}$, we have $N(r,x^{\otimes c})=0$ and
$$
T(r,x^{\otimes c})=m(r,x^{\otimes c})=\frac{|c|}{2}r \quad (r>0).
$$
\end{proposition}

We next observe that a result reminiscent to the behavior of
elliptic functions in the complex plane holds for tropical periodic
functions:
\begin{proposition} \label{prop:1}
A non-constant tropical periodic function has as many roots and
poles in a period interval, counting multiplicities. In other words,
the sum
$$
\sum_{c\in (\text{supp}~\omega_f)\cap[0,1)}\omega_f(c) =\sum_{a\in
[0,1), \ \omega_f(a)>0}\tau_f(a)-\sum_{b\in [0,1), \
\omega_f(b)<0}\tau_f(b)
$$
vanishes for any tropical meromorphic $1$-periodic function $f(x)$.
More generally, any tropical meromorphic function $g$ with the same
value at two different points $x=x_0$ and $x=x_1$ satisfies
$$
\sum_{x_0\leq x <x_1} \omega_g(x)=0\,,
$$
so that $g$ has the same number of poles and roots, counting
multiplicities, in the interval $[x_0,x_1)$,
$$
\sum_{a\in [x_0,x_1), \ \omega_g(a)>0}\tau_g(a)=\sum_{b\in
[x_0,x_1), \ \omega_g(b)<0}\tau_g(b)\,.
$$
\end{proposition}

\begin{proof} In fact, any nonconstant tropical meromorphic $1$-periodic
function $f(x)$ has the same value at $x=0$ and $x=1$. Let
$0<c_1<\ldots <c_K<1$ be all the elements of the set
$(\text{supp}~\omega_f)\cap (0,1)$. Then we have
\begin{eqnarray*}
\sum_{j=1}^K \omega_f(c_j) &=&\sum_{j=1}^K\left\{
f'\left(\frac{c_{j+1}+c_{j}}{2}\right)
                      -f'\left(\frac{c_j+c_{j-1}}{2}\right)\right\}\\
&=&f'\left(\frac{c_{K+1}+c_K}{2}\right)-f'\left(\frac{c_1+c_0}{2}\right)\,,
\end{eqnarray*}
where we put $c_0=0$ and $c_{K+1}=1$. It is not difficult to show
that the number on the right-hand side is equal to $-\omega_f(0)$
since $f'\bigl((c_{K+1}+c_{K})/2\bigr)=f'\bigl((0+c_{K}-1)/2\bigr)$
by the $1$-periodicity of $f$. Therefore, the assertion follows,
independently of whether $0\in \text{supp}~\omega_f$ or not.
\end{proof}

As a solution of equation $f(x+1)=f(x)$, each tropical meromorphic
$1$-periodic function can be explicitly constructed in the following
manner.

\medskip

We first consider a simple tropical meromorphic $1$-periodic
function defined by
\begin{eqnarray*}
\pi^{(a,b)}(x) &:=& \frac{1}{a+b} \max \left\{ \, a(x-[x]), -b((x-[x])-1) \, \right\}\\
         &=& \left\{\bigl(a(x-[x])\bigr) \oplus \bigl(-b(x-[x])-1)\bigr)\right\}^{\otimes 1/(a+b)}
\end{eqnarray*}
for arbitrary parameters $a, b\in\mathbb{R}_{<0}$ (or `$\min$' for
$a, b\in\mathbb{R}_{>0}$). Here $[x]$ denotes the floor function of
$x$, that is, the greatest integer which does not exceed the value
$x$. For example, $\pi^{(-1,-1)}(x)=\frac{1}{2}\min\{x-[x],
(-x)-[(-x)]\}$, which has an isosceles sawtooth waveform with
width~$1$ and height ~$1/4$. Note that
$$
\pi^{(a,b)}(n)=0 \quad \text{and} \quad
\pi^{(a,b)}\left(\frac{b}{a+b}+n\right)
     =\frac{ab}{(a+b)^2},
$$
$$
\omega_{\pi^{(a,b)}}(n)=1 \quad \text{and}  \quad
\omega_{\pi^{(a,b)}}\left(\frac{b}{a+b}+n\right)=-1
$$
for any integer $n$, while $\omega_{\pi^{(a,b)}}(x)=0$ otherwise.
This $\pi^{(a,b)}(x)$ is a tropical meromorphic $1$-periodic
function with the single simple zero at $0\mod 1$ and the single
simple pole at $b/(a+b)\mod 1$ in each periodic interval. For each
value $c<ab/(a+b)^2$, the function $\pi_c^{(a,b)}(x):=1\oslash
\bigl(\pi^{(a,b)}(x)\oplus c \bigr)=-\max\{\pi^{(a,b)}(x), c\}$ has
the single simple zero at $b/(a+b)\mod 1$ with
$$
\omega_{\pi_c^{(a,b)}(x)}\left(\frac{a}{a+n}+n\right)=1 \quad (n \in
\mathbb{Z})
$$
and two poles at $c\frac{a+b}{a}\mod 1$ and $1-c\frac{a+b}{b}\mod 1$
in each periodic interval with
$$
\omega_{\pi_c^{(a,b)}(x)}\left(c\frac{a+b}{a}+n\right)=-\frac{a}{a+b}
\quad \text{and}
$$
$$
\omega_{\pi_c^{(a,b)}(x)}\left(1-c\frac{a+b}{b}+n\right)=-\frac{b}{a+b}
\quad (n\in\mathbb{Z})\,,
$$
respectively.

\medskip

In general, any non-constant tropical meromorphic $1$-periodic
function $f(x)$ can be represented as an $\mathbb{R}$-linear
combination of such functions $\pi^{(a,b)}(x)$. In fact, without
loss of generality, we may assume that the support of $\omega_f(x)$
in the interval $[0,1)$ consists of $K(\geq 1)$ points $0<c_1<\cdots
<c_K <1$. This finiteness comes from the property of piecewise
linearity of our functions, and this implies the fact that they are
of growth order $2$.

Then take the parameters $(a_k,b_k)$ so that $c_k=b_k/(a_k+b_k)$ for
$k=1,\ldots, K$ and define
\begin{equation} \label{func:1}
\hat{f}(x)=-\sum_{k=1}^K \omega_f(c_k) \pi^{(a_k,b_k)}(x)+f(0)\,.
\end{equation}

Now we note any finite sum of tropical meromorphic functions is
again a tropical meromorphic function and $\omega_f(x)$ has the
linearity property like $\omega_{\sum_j A_jf_j}(x)=\sum_j
A_j\omega_{f_j}(x)$ for the real constants $A_j$. Further we see

\begin{proposition} \label{prop:2}
Two tropical meromorphic functions $f(x)$ and $g(x)$ on a closed
interval $\Delta\subset\mathbb{R}$ satisfy the relation
$$
\omega_f(x)-\omega_g(x)\equiv 0
$$
if and only if $f(x)-g(x)$ is a linear function on $\Delta$.
\end{proposition}


\begin{proof} If a tropical meromorphic function $h(x)$ satisfies
$\omega_h(x)\equiv 0$ on $\mathbb{R}$, then by definition
$h'(x+)=h'(x-)$ holds for any $x\in\mathbb{R}$, that is, $h(x)$ is
differentiable on the whole $\mathbb{R}$ since it is continuous on
the whole $\mathbb{R}$. Then the assumption $\omega_h(x)\equiv 0$
implies the derivative $h'(x)$ of the tropical meromorphic function
$h(x)$ is a constant on $\mathbb{R}$. Of course, the $\omega$ of any
linear function vanishes identically.
\end{proof}

Returning to the tropical meromorphic $1$-periodic function
$\hat{f}(x)$ in (\ref{func:1}), we observe that
$$
\omega_{\hat{f}}(c_{\ell}) =-\omega_f(c_{\ell})\pi^{(a_{\ell},
b_{\ell})}(c_{\ell}) =\omega_f(c_{\ell})\,, \quad 1\leq \ell \leq K
$$
and further by Proposition~\ref{prop:1},
$$
\omega_{\hat{f}}(0) =-\sum_{k=1}^K \omega_f(c_{k})\pi^{(a_{k},
b_{k})}(0) =-\sum_{k=1}^K \omega_f(c_{k}) =0=\omega_f(0)\,,
$$
so that the relation $\omega_f(x)-\omega_{\hat{f}}(x)\equiv 0$ holds
on $\mathbb{R}$. Proposition~\ref{prop:2} implies
$f(x)=\hat{f}(x)+Ax+B$ for some real constants $A$ and $B$, which
are indeed determined as $B=0$ and
$$
A=\frac{f(c_1)-f(0)}{c_1}+(1-c_1)\omega_f(c_1)
$$
by using $\hat{f}(0)=f(0)$ and
\begin{eqnarray*}
\hat{f}(c_1)&=&\hat{f}\left(\frac{b_1}{a_1+b_1}\right)
= -\omega_f(c_1) \pi^{(a_1,b_1)} \left(\frac{b_1}{a_1+b_1}\right)+f(0)\\
&=& -\omega_f(c_1)\frac{a_1b_1}{(a_1+b_1)^2}+f(0) =
c_1(c_1-1)\omega_f(c_1)+f(0)\,,
\end{eqnarray*}
respectively. Hence the representation formula for the non-constant
tropical meromorphic $1$-periodic function $f(x)$ must be
$$
f(x) =-\sum_{k=1}^K \omega_f(c_k)
\pi^{(a_k,b_k)}(x)+(1-c_1)\omega_f(c_1)x
+\frac{f(c_1)-f(0)}{c_1}x+f(0)\,.
$$

\begin{example} In the above expression, taking $K=1$, $(a_1,b_1)=(-1,-1)$, $(a_2,b_2)=(-1,-2)$ and
thus $c_1=1/2$, $c_2=2/3$, and further $\omega_f(c_1)=1/3$,
$\omega_f(c_2)=-1/2$ and so on, we have the tropical meromorphic
$1$-periodic function
\begin{eqnarray*}
f(x)&=&-\frac{1}{3}\pi^{(-1,-1)}(x)+\frac{1}{2}\pi^{(-1,-2)}(x)\\
&=& \left\{
\begin{array}{cl}
0 & (0\leq x-[x] \leq 1/2)\,, \\
\frac{1}{3}(x-[x])-\frac{1}{6} & (1/2\leq (x-[x]) \leq  2/3)\,, \\
-\frac{1}{6}(x-[x])+\frac{1}{6} & (2/3\leq (x-[x]) \leq 1)\,.
\end{array}
\right.
\end{eqnarray*}
Thus this $f(x)$ satisfies an equation of the form
$$
\max \bigl\{ y-\max\{0, \frac{1}{3}(x-[x])-\frac{1}{6}\}\,, \
                  y-\max\{0, -\frac{1}{6}(x-[x])+\frac{1}{6} \}  \bigr\}=0
                  \quad \text{or}
$$
\begin{eqnarray*}
& & \max \bigl\{ 6y-\max\{0, 2(x-[x])-1\}\,, \
                  6y-\max\{0, -(x-[x])+1 \}  \bigr\} \\
&=& \max \bigl\{ 6y-1-\max\{-1, 2(x-[x]) \}\,, \
                  6y+1-\max\{-1, -(x-[x]) \}  \bigr\} \\
&=&0\,,
\end{eqnarray*}
which may be written as an ultra-discrete equation
$$
\bigg\{ (-1)\otimes y^{\otimes 6}\oslash \Bigl((-1)\oplus
(x-[x])^{\otimes 2}\Bigr) \bigg\} \oplus \bigg\{ 1\otimes y^{\otimes
6} \oslash \Bigl((-1)\oplus (x-[x])^{\otimes (-1)}\Bigr) \bigg\}
=0\,.
$$
\end{example}

Finally, we see

\begin{proposition} \label{prop:3} Any non-constant tropical periodic
meromorphic function $f$ satisfies $T(r,f)\asymp \kappa r^{2}$ for
some $\kappa>0$, hence is of order two.
\end{proposition}

\begin{proof} Due to Proposition~\ref{prop:1} as well as
Proposition~\ref{prop:0}, we need only to prove $T(r,
\pi^{(a,b)})=\frac{r^2}{2}+O(r)$. Since $n(r, -
\pi^{(a,b)})=2[r]+1$, we have $N(r,- \pi^{(a,b)})=\int_0^r
[t]dt+r/2\sim r^2/2+O(r)$. On the other hand,
$$
m(r, - \pi^{(a,b)})\leq |\pi^{(a,b)}(r)|+|\pi^{(a,b)}(-r)|\leq
\frac{2ab}{(a+b)^2}\leq \frac{1}{2}\,.
$$
\end{proof}

\section{Tropical counterpart to the exponential
function}\label{tropexp}

We start our construction by considering certain special tropical
meromorphic functions, which are reminiscent to exponential
functions over the usual algebra.

\begin{definition} \label{def:1}
Let $\alpha$ be a real number with $|\alpha|>1$. Define a function
$e_{\alpha}(x)$ on $\mathbb{R}$ by
$$
e_{\alpha}(x):=\alpha^{[x]}(x-[x])+\sum_{j=-\infty}^{[x]-1} \alpha^j
=\alpha^{[x]}\left(x-[x]+\frac{1}{\alpha-1}\right)\,.
$$
\end{definition}

Then we see

\begin{proposition}\label{alfa}
The function $e_{\alpha}(x)$ is tropical meromorphic on $\mathbb{R}$
satisfying
\begin{itemize}
\item $e_{\alpha}(m)=\alpha^{m}/(\alpha-1)$ for each $m\in\mathbb{Z}$,
\item $e_{\alpha}(x) = x+\frac{1}{\alpha-1}$ for any $x\in [0,1)$, and
\item the functional equation $y(x+1)=y(x)^{\otimes \alpha}$ on the whole $\mathbb{R}$.
\end{itemize}
\end{proposition}

\begin{proof} The two first assertions trivially follow from Definition
\ref{def:1}. The last assertion is verified by a straightforward
computation:
$$
e_{\alpha}(m) =\sum_{j=-\infty}^{m-1} \alpha^j
=\alpha^{m-1}\sum_{j=-\infty}^0 \alpha^j
=\alpha^{m-1}\sum_{j=0}^{\infty} \left(\frac{1}{\alpha}\right)^j
=\frac{\alpha^m}{\alpha-1}
$$
when $m\in\mathbb{Z}$,
$$
e_{\alpha}(x)=\sum_{j=-\infty}^{-1} \alpha^j +x =
x+\frac{1}{\alpha-1}
$$
when $x\in [0,1)$, and on $\mathbb{R}$
\begin{eqnarray*}
e_{\alpha}(x+1)&=& \sum_{j=-\infty}^{[x+1]-1} \alpha^j +\alpha^{[x+1]}(x+1-[x+1])\\
                   &=& \sum_{j=-\infty}^{[x]} \alpha^j +\alpha^{[x]+1}(x-[x])\\
                   &=& \alpha\left(\sum_{j=-\infty}^{[x]-1}\alpha^j
                         +\alpha^{[x]}(x-[x])\right)=\alpha\, e_{\alpha}(x)\,.
\end{eqnarray*}
It remains to verify that $e_{\alpha}(x)$ is continuous at integer
points $x=m\in\mathbb{Z}$. This follows by taking $\varepsilon\in
(0,1)$ and and observing that
\begin{eqnarray*}
e_{\alpha}(m+\varepsilon) &=&\sum_{j=-\infty}^{m-1}\alpha^j +
\alpha^{m}(m+\varepsilon-m)
=\frac{\alpha^m}{\alpha-1}+\varepsilon \alpha^{m}\quad \text{and}\\
e_{\alpha}(m-\varepsilon) &=&\sum_{j=-\infty}^{m-1-1}\alpha^j +
\alpha^{m-1}(m-\varepsilon-m+1)
=\frac{\alpha^m}{\alpha-1}-\varepsilon \alpha^{m-1}\,.
\end{eqnarray*}
\end{proof}

In a similar way, for a real number with $|\beta|<1$, we consider a
function
\begin{eqnarray*}
e_{\beta}(x)&:=&\sum_{j=[x]}^{\infty} \beta^j -\beta^{[x]}(x-[x])
=\sum_{j=[x]+1}^{\infty} \beta^j +\beta^{[x]}(1-x+[x])\\
&=&\beta^{[x]}\left(\frac{1}{1-\beta}-x+[x]\right)\,.
\end{eqnarray*}
Then we similarly obtain
\begin{proposition}\label{beeta}
This function $e_{\beta}(x)$ is also tropical meromorphic on
$\mathbb{R}$ and satisfies
\begin{itemize}
\item $e_{\beta}(m)=\beta^{m}/(1-\beta)$ for each $m\in\mathbb{Z}$,
\item $e_{\beta}(x) = -x+\frac{1}{1-\beta}$ for any $x\in [0,1)$, and
\item the functional equation $y(x+1)=y(x)^{\otimes \beta}$ on the whole $\mathbb{R}$.
\end{itemize}
\end{proposition}

\begin{proof} Again, the last assertion is the only to be checked:
\begin{eqnarray*}
e_{\beta}(x+1)&=& \sum_{j=[x+1]}^{\infty} \beta^j -\beta^{[x+1]}(x+1-[x+1])\\
              &=& \beta\left(\sum_{j=[x]}^{\infty} \beta^j -\beta^{[x]}(x-[x]) \right)
                    =\beta\, e_{\beta}(x)\,.
\end{eqnarray*}
As for the continuity at integer points $x=m\in\mathbb{Z}$, we may
take $\varepsilon\in (0,1)$, resulting in
\begin{eqnarray*}
e_{\beta}(m+\varepsilon) &=&\sum_{j=m}^{\infty}\beta^j -
\beta^{m}(m+\varepsilon-m)
=\frac{\beta^{m}}{1-\beta}-\varepsilon \beta^{m}\quad \text{and}\\
e_{\beta}(m-\varepsilon) &=&\sum_{j=m-1}^{\infty}\beta^j -
\beta^{m-1}(m-\varepsilon-m+1)
=\frac{\beta^{m}}{1-\beta}+\varepsilon \beta^{m-1}\,.
\end{eqnarray*}
\end{proof}

As for the connection between $e_{\alpha}(x)$ with $|\alpha|>1$ and
$e_{\beta}(x)$ with $|\beta|<1$, we obtain

\begin{proposition}\label{alfabeta} Suppose $\alpha\neq\pm 1$. Then
\begin{itemize}
\item $e_{\alpha}(-x)=\frac{1}{\alpha}e_{1/\alpha}(x)$, and
\item $e_{\alpha}(0)=\frac{1}{\alpha}e_{1/\alpha}(0).$
\end{itemize}
\end{proposition}

\begin{proof} The first assertion immediately follows from the
expressions for $e_{\alpha}(x)$ and $e_{1/\alpha}(x)$ and from
$[-x]=-[x]-1$. The second assertion is trivial.
\end{proof}

\bigskip

\begin{remark} The slopes of $e_{\alpha}(x)$ range over the infinite set $\{
\alpha^j \, | \, -\infty<j<+\infty\}$ for all $\alpha\neq\pm 1$.
\end{remark}

\begin{proposition}\label{order} The function $e_{\alpha}(x)$, $\alpha\neq\pm 1$, is of infinite order
 and, in fact, of hyper-order one.
\end{proposition}

\begin{proof} If $\alpha>1$, then $e_{\alpha}(x)$ is strictly positive and has no
poles. Moreover, for $r=m+\varepsilon$ with $m\in\mathbb{Z}$ and
$\varepsilon\in [0, 1)$, we have
\begin{eqnarray*}
T(r,e_{\alpha})=m(r, e_{\alpha}) &=&
\frac{\alpha^{m}\bigl(\varepsilon+1/(\alpha-1)\bigr)
             +\alpha^{-m-1}\bigl(1-\varepsilon+1/(\alpha-1)\bigr)}{2}\\
&=& \frac{1}{2}\left(\frac{1}{\alpha-1}+r-[r]+o(1) \right)
\alpha^{[r]}\,.
\end{eqnarray*}
Therefore, $e_{\alpha}(x)$ is of hyper-order one:
$$
\limsup_{r\to\infty} \frac{\log\log T(r, e_{\alpha})}{\log
r}=\lim_{r\to\infty}\frac{\log [r]+O(1)}{\log r}=1.
$$

\medskip

The case $\alpha <1$ immediately follows from Proposition
\ref{alfabeta}.
\end{proof}

\begin{remark} Observe that if $\alpha >1$, then
$e_{\alpha}(x) \oplus a >0$ for any $a\in\mathbb{R}$. This shows
that $m\bigl(r, 1_{\circ}\oslash (e_{\alpha}\oplus
a)\bigr)=m\bigl(r, -(e_{\alpha}\oplus a)\bigr)\equiv 0$, so that
$T\bigl(r, 1_{\circ}\oslash (e_{\alpha}\oplus a)\bigr)=N\bigl(r,
1_{\circ}\oslash (e_{\alpha}\oplus a)\bigr)$.

\medskip

On the other hand, if $\alpha<-1$, then $e_{\alpha}(x)$ has a zero
of multiplicity $\alpha^{2j}(1-1/\alpha)$ at each even integer
$x=2j$ and a pole of multiplicity $\alpha^{2j}(1-\alpha)$ at each
odd integer $x=2j+1$, since
$\omega_{e_{\alpha}}(m)=\alpha^{m}(1-1/\alpha)$ for each
$m\in\mathbb{Z}$. Thus we see that when $2\ell\leq t<2(\ell+1)$ for
some integer $\ell$, that is, when $\ell= \left[\frac{t}{2}\right]$,
then
\begin{eqnarray*}
n(t, 1_{\circ}\oslash e_{\alpha}) &=&\sum_{j=-\ell}^{\ell}
\alpha^{2j}\left(1-\frac{1}{\alpha}\right)
=\sum_{j=1}^{\ell} \left(\alpha^{2j}+\alpha^{-2j}\right)\left(1-\frac{1}{\alpha}\right)+1-\frac{1}{\alpha}\\
&=&\left(1-\frac{1}{\alpha}\right)\left\{\frac{\alpha^2}{\alpha^2-1}\alpha^{2\ell}+\frac{1}{1-\alpha^2}\alpha^{-2\ell}\right\}\\
&=&\frac{\alpha}{\alpha+1}\alpha^{2[t/2]}-\frac{1}{\alpha(\alpha+1)}\alpha^{-2[t/2]}\\
&\geq
&\frac{1}{\alpha(\alpha+1)}|\alpha|^{t}-\frac{\alpha}{\alpha+1}|\alpha|^{-t}
\end{eqnarray*}
so that $N(r, e_{\alpha})\geq
|\alpha|^r/\{2\alpha(\alpha+1)\log\alpha\}+O(1)$.
\medskip
As for the case of $e_{\beta}(x)$, its slopes again range over the
infinite set $\{ \beta^j \, | \, -\infty<j<+\infty\}$. When
$0<\beta<1$, this function has no poles and $m(r, e_{\beta})\asymp C
\beta^{-r}$. In the case of $-1<\beta<0$, it has a pole of
multiplicity $\beta^{2j}(1-1/\beta)$ at each even integer $x=2j$ and
a root of multiplicity $\beta^{2j}(1-\beta)$ at each odd integer
$x=2j+1$. In particular, taking $\beta =-1/2$, it is immediate to
obtain that
$$N(r,e_{\beta}(x))=2N(r,r,1_{\circ}\oslash e_{\beta}(x))+O(r).$$

\medskip

In order to see that the assertion of Corollary \ref{nodef} fails
for $e_{\beta}(x)$ with $\beta =-1/2$, see Figure 1 below, take
$a=-1<0=L_{e_{\beta}}$. Then the roots of $e_{\beta}(x)\oplus a$ are
the same as those of $e_{\beta}(x)$ for all $x=2j+1>0$, while for
$x=2j+1<0$, each such root of $e_{\beta}(x)$, having multiplicity
$\beta^{2j}(1-\beta)$, splits into two roots of $e_{\beta}(x)\oplus
a$, with the sum of their multiplicities being equal to
$\beta^{2j}(1-\beta)$, see Figure 1 and Figure 2 below. Therefore,
we have
\begin{equation}\label{defja}
T(r,e_{\beta}(x))\geq
N(r,e_{\beta}(x))=2N(r,1_{\circ}\oslash(e_{\beta}(x)\oplus a))+O(r).
\end{equation}

\begin{figure}[h!]
  \begin{center}
    \includegraphics{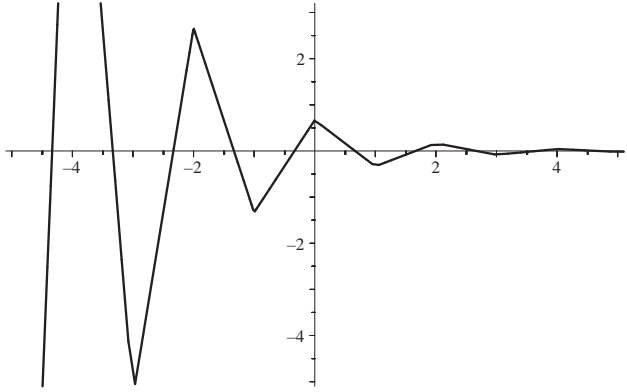}
  \end{center}
  \caption{Function $e_{-1/2}(x)$.}
  \label{fig:ehalf}
\end{figure}

\begin{figure}[h!]
  \begin{center}
    \includegraphics{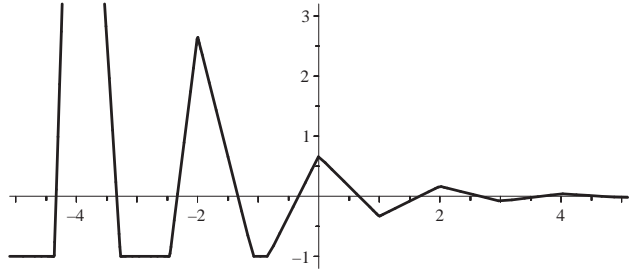}
  \end{center}
  \caption{Function $e_{-1/2}(x) \oplus (-1)$.}
  \label{fig:sols}
\end{figure}

More generally, the same conclusion as in (\ref{defja}) follows for
all $a<0$. In particular, this means that each $a<0$ is a deficient
value for $e_{\beta}(x)$ in the sense that
$$1-\limsup_{r\rightarrow}\frac{N(r,1_{\circ}\oslash (e_{\beta}(x)\oplus a))}{T(r,e_{\beta}(x))}\geq 1/2>0.$$
\end{remark}

\begin{remark} To prevent misinterpretations in what
follows, let $e_{\alpha}(x)$ and $e_{\beta}(x)$ be two tropical
exponential functions with $\alpha,\beta\neq 1$, $\alpha\neq\beta$
and let $s\in [0,1)$ be a fixed real number. Then it is immediate to
verify that we get for the Casoratian determinants
$$\left|\begin{array}{cc}
      e_{\alpha}(x) & e_{\alpha}(x-s) \\
      e_{\alpha}(x+1) & e_{\alpha}(x-s+1)
    \end{array}\right|= 0,
$$
and
$$\left|\begin{array}{cc}
      e_{\alpha}(x) & e_{\beta}(x) \\
      e_{\alpha}(x+1) & e_{\beta}(x+1)
    \end{array}\right|\neq 0.
$$
However, one of the tropical exponentials in the pairs
$e_{\alpha}(x),e_{\alpha}(x-s)$ and $e_{\alpha}(x),e_{\beta}(x)$ is
not a constant multiple of the other one. Recall, however, that
$\alpha e_{\alpha}(-x)=e_{\beta}(x)$ by Proposition \ref{alfabeta}
whenever $\beta=1/\alpha$. Therefore, the standard presentation of
linear (in)dependence of linear difference equations as in \cite{E},
say, does not carry over as such to the tropical setting of
meromorphic functions.
\end{remark}

\section{An application to ultra-discrete equations: first order}\label{infinite}

In this section, we consider ultra-discrete equations of type
\begin{equation} 
y(x+1)=y(x)^{\otimes c} =c\, y(x) \qquad (c\in\mathbb{R}).
\end{equation}
First restricting ourselves to the case
\begin{equation} \label{eqn:2}
y(x+1)=y(x)^{\otimes n} =n\, y(x) \qquad (n\in\mathbb{Z}),
\end{equation}
the existence of non-constant tropical meromorphic solutions depends
on the value of $n\in\mathbb{Z}$. In fact, Halburd and Southall
\cite{HS}, Lemma~4.1, have shown that equation (\ref{eqn:2}) admits
a non-constant tropical meromorphic solution (in the case of integer
slopes) on $\mathbb{R}$ if and only if $n=\pm 1$. As for the case of
real slopes, we now proceed to proving

\begin{theorem} \label{prop:first-order} Equation $y(x+1)=y(x)^{\otimes c}$
with $c\in\mathbb{R}\setminus\{0\}$ admits a non-constant tropical
meromorphic function on $\mathbb{R}$ of hyper-order $<1$ if and only
if $c=\pm 1$. If $c=1$, resp. $c=-1$, non-constant tropical
meromorphic solutions to (\ref{eqn:c}) are $1-$periodic, resp.
$2-$periodic, anti-$1$-periodic, tropical meromorphic functions.
Given an arbitrary tropical meromorphic solution $f$ to
(\ref{eqn:c}) with discontinuities of slope at $x_{1},\ldots ,x_{q}$
in $[0,1)$, if $c\neq 0, \pm 1$, then $f$ may be represented as
$$
f(x)=\rho(c)\sum_{j=1}^q \omega_f(x_j) e_{c}(x-x_j) \quad
\text{with} \ \rho(c)= \left\{
\begin{array}{cc}
c-1 & (|c|>1), \\
1-c & (0<|c|<1)\,,
\end{array}
\right.
$$
\end{theorem}

\begin{proof} First considering equation~(\ref{eqn:2}) with $n\in\mathbb{Z}$,
$e_n(x)$ is one of the solutions when $|n|>1$. On the other hand, as
$1$-periodic functions, all tropical meromorphic solutions of
$y(x+1)=y(x)$ have been determined. Since equation (\ref{eqn:2})
implies that $y(x+2)=(-1)^2y(x)=y(x)$, all non-trivial tropical
meromorphic solutions of $y(x+1)=-y(x)$ are given by
$y(x)=u(x+1)-u(x)$, where $u(x)$ is a $2$-periodic (but not
$1$-periodic) function. Hence, if $n=\pm 1$, all non-constant
tropical meromorphic solutions of (\ref{eqn:2}) are of order two.

\bigskip

We now proceed to giving the asserted representation of an arbitrary
tropical meromorphic solution $f$ to equation
\begin{equation} \label{eqn:c}
y(x+1)=y(x)^{\otimes c} =c\, y(x)
\end{equation}
for $c\in\mathbb{R}$ with $c \neq 0, \pm 1$ as a linear combination
of finite shifts of the function $e_{c}(x)$ over $\mathbb{R}$.
Indeed, given a non-trivial tropical meromorphic solution $f$ to
equation~(\ref{eqn:c}), there are only finitely many points $x_1,
\ldots , x_q$ in the interval $[0,1)$ on which $\omega_f(x)\neq 0$.
Define now a tropical meromorphic function as
$$
g(x):=\rho(c)\sum_{j=1}^q \omega_f(x_j) e_{c}(x-x_j) \quad
\text{with} \ \rho(c)= \left\{
\begin{array}{cc}
c-1 & (|c|>1), \\
1-c & (0<|c|<1)\,.
\end{array}
\right.
$$
Clearly, $g$ is a solution to equation (\ref{eqn:c}):
\begin{eqnarray*}
g(x+1)&=&\rho(c) \sum_{j=1}^q \omega_f(x_j)e_{c}(x+1-x_j)\\
        &=&c \left(\rho(c) \sum_{j=1}^q \omega_f(x_j) e_c(x-x_j)\right)\\
        &=&c\, g(x)\,.
\end{eqnarray*}
When $x \in [0,1)$, we see that $\omega_g(x)=0$ if $x\neq x_j$ and
$$\omega_g(x_j)=\omega_f(x_j)\rho(c) e_{c}(0)=\omega_f(x_j).$$
Making use of equation (\ref{eqn:c}) satisfied by both $f$ and $g$,
it is immediate to see that $\omega_g(x)\omega_f(x)$ holds for all
$x\in\mathbb{R}$. Therefore, we may apply Proposition~\ref{prop:2}
to conclude that $f(x)=g(x)+Ax+B$ for some real constants $A,
B\in\mathbb{R}$. By
$f(x+1)-cf(x)=g(x+1)+A(x+1)+B-cg(x)-cAx-cB=(1-c)Ax+A+(1-c)B=0$, we
conclude that $A=0$ and $B=0$, since $c\neq 1$. Therefore, $f=g$ on
the interval $[0,1)$.
\end{proof}

\begin{remark} When applying Theorem \ref{prop:first-order} in the
next section, we prefer to write the representation for $f$ in the
form
$$f(x)=\sum_{j=1}^{q}a_{j}e_{c}(x-x_{j}).$$
\end{remark}

\section{An application to ultra-discrete equations: second order}\label{second}

In this section,we are considering equation
\begin{equation} \label{eqn:3}
y(x+1)\otimes y(x-1) =y(x)^{\otimes c} \quad \text{i.e.} \quad
y(x+1)+y(x-1)=c\, y(x)
\end{equation}
for $c\in\mathbb{R}$. In the restricted setting of integer slopes,
the considerations in \cite{HS} show that whenever $c\in\mathbb{Z}$,
(\ref{eqn:3}) admits tropical meromorphic solutions of finite order
(with integer slopes) if and only if $c=0, \pm 1, \pm 2$. Observe in
particular, that the paper \cite{LY2} claiming that $c=\pm 2$ only
is possible, contains a slip in the reasoning. In our more general
setting of real slopes, tropical meromorphic solutions of finite
order exist for all $c\leq 2$, while for $c>2$, tropical meromorphic
solutions are of hyper-order $\rho_{2}=1$.

\medskip

In order to prove the existence of tropical meromorphic solutions to
equation (\ref{eqn:3}) and to obtain a representation for all of
them, let $a,b$ be the roots of $\lambda^{2}-c\lambda +1=0$. Then we
have, of course, $a+b=c$ and $ab=1$. We first prove

\begin{theorem}\label{prop:second} Tropical meromorphic solutions $f$
of equation (\ref{eqn:3}) with $|c|\geq 2$ exist and may be
represented in the following forms:

\medskip

(i) If $c=2$, then $f$ must be a linear combination of $L(x)=x$ and
a $1$-periodic tropical meromorphic function $\Pi_{1}(x)$.

\medskip

(ii) If $c=-2$, then
$$f(x)=\sum_{j=1}^{s}(-1)^{[x-x_{j}]}\Xi_{1,j}(x),$$
where $\Xi_{1,j}(x)$ is a $1$-periodic function such that
\begin{equation}\label{minustwo}
\lim_{\varepsilon\rightarrow
0}((-1)^{[n-\varepsilon-x_{j}]}\Xi_{1,j}(n-\varepsilon-x_{j})-(-1)^{[n+\varepsilon-x_{j}]}\Xi_{1,j}(n+\varepsilon-x_{j}))=0
\end{equation}
holds for all $n\in\mathbb{Z}$ and all $j=1,\ldots ,s$. Here
$x_{j}$:s are the slope discontinuities of $f$ in the interval
$[0,2)$.

\medskip

(iii) If $|c|>2$, then
$$f(x)=\sum_{j=1}^{p}\alpha_{j}e_{a}(x-y_{j})+\sum_{j=1}^{q}\beta_{j}e_{a}(-x+x_{j}),$$
where $y_{j}$, resp. $x_{j}$, are the points of slope discontinuity
of $f$ in the interval $[0,1)$, resp. in $[0,2)$.
\end{theorem}

\begin{proof} To start the proof, observe that equation
(\ref{eqn:3}) may be written in the form

\begin{equation}\label{3m}
y(x+1)-ay(x)=b(y(x)-ay(x-1)).
\end{equation}

We first consider the case $c=2$. Then $a=b=1$. Thus, equation
(\ref{3m}) now is
$$
y(x+1)-y(x)=y(x)-y(x-1).
$$
This equation has two classes of tropical meromorphic solutions: All
$1$-periodic tropical meromorphic functions that satisfy equation
$y(x)-y(x-1)=0$, and tropical meromorphic solutions of
$y(x)-y(x-1)=\Pi_{1}(x)$ with a $1$-periodic function $\Pi_{1}(x)$.
The latter class consists of $x\Pi_{1}(x)+\Xi_{1}(x)$, where
$\Xi_{1}(x)$ is an arbitrary $1$-periodic tropical meromorphic
function. Since tropical meromorphic functions are continuous and
piecewise linear, $\Pi_{1}(x)$ in the latter has to a constant.
Therefore, the general solution of (\ref{eqn:3}) with $c=2$ consists
of linear functions, of $1$-periodic tropical meromorphic functions
and of their linear combinations.

\bigskip

When $c=-2$, then equation (\ref{3m}) is
$$
y(x+1)+y(x)=-\{y(x)+y(x-1)\}.
$$
In order to construct solutions of this equation, we first observe
that
$$g(x):=y(x)+y(x-1)$$
has to be anti-$1$-periodic: $g(x+1)=-g(x)$. Therefore, $g(x)$ may
be written in the form
$$g(x)=\sum_{j=1}^{p}(-1)^{[x-x_{j}]}\Xi_{1,j}(x),$$
where each $\Xi_{1,j}(x)$ is a $1$-periodic function and the
$x_{j}$:s are the slope discontinuities of $y$ in the interval
$[0,2)$. If $g$ vanishes, then $y$ itself may be written in the same
form
$$y(x)=\sum_{j=1}^{p}(-1)^{[x-x_{j}]}\Xi_{1,j}(x).$$
Moreover, since $y$ has to be tropical meromorphic, hence
continuous, $\Xi_{1}(x)$ has to satisfy the condition
(\ref{minustwo}) for all $n\in\mathbb{Z}$. If $g$ does not vanish,
then we obtain have
$$y(x+1)+y(x)=\sum_{j=1}^{p}(-1)^{[x-x_{j}]}\Xi_{1,j}(x).$$
By linearity, $y(x)$ is then the sum of a solution of the
homogeneous equation $y(x)+y(x-1)$, already treated, and special
solutions $y_{j}(x)$ of
$$y_{j}(x+1)+y_{j}(x)=(-1)^{[x-x_{j}]}\Xi_{1,j}(x), \quad j=1,\ldots ,p.$$
But these solutions $y_{j}(x)$ may now be written in the form
$$y_{j}(x)=-x(-1)^{[x-x_{j}]}\Xi_{1,j}(x).$$
where $\Xi_{1,j}(x)$:s are $1$-periodic functions. Since each
$y_{j}$ has to be piecewise linear, we observe that $\Xi_{1,j}$ has
to be a constant, and in fact equal to zero by the continuity of
$y_{j}$. Therefore, indeed, $g(x)$ has to vanish identically, and we
are done with the case $c=-2$.

\bigskip

Assuming now that $|c|>2$, we have that $a,b$ are real and distinct.
We may assume that $a>1$. Denoting now
\begin{equation}\label{gequ}
g(x):=y(x)-ay(x-1),
\end{equation}
we see from (\ref{3m}) that $g$ satisfies
$$g(x+1)-bg(x)=0.$$
By Theorem \ref{prop:first-order}, $g(x)$ has the representation
\begin{equation}\label{grep}
g(x)=\sum_{j=1}^{q}\beta_{j}e_{b}(x-x_{j}),
\end{equation}
where the slope discontinuities of $g(x)$ in the interval $[0,1)$
are at $x_{1},\ldots ,x_{q}$. Clearly, these are nothing but the
discontinuities of slopes of $y(x)$ in the interval $[0,2)$. It may
happen that some slope discontinuities of $y(x)$ in $[0,1)$ and
$[1,2)$ cancel for $g(x)$. To avoid complicated notations, such
cases are being included in the sum (\ref{grep}) with a zero
coefficient $\beta_{j}$.

\medskip

We next determine special tropical solutions of
\begin{equation}\label{spec}
y_{j}(x+1)-ay_{j}(x)=b\beta_{j}e_{b}(x-x_{j})
\end{equation}
for each $j=1,\ldots ,q$. These are immediately found by
substituting $y_{j}(x)=A_{j}e_{b}(x-x_{j})$ in (\ref{spec}) to
determine the constants $A_{j}=b\beta_{j}/(b-a)$. Therefore,
$$y_{0}(x)=\sum_{j=1}^{q}\frac{b\beta_{j}}{b-a}e_{b}(x-x_{j})$$
satisfies
$$y_{0}(x+1)-ay_{0}(x)=bg(x)=\sum_{j=1}^{q}b\beta_{j}e_{b}(x-x_{j}).$$
But now, the difference $y_{H}(x):=y(x)-y_{0}(x)$ satisfies
$$y_{H}(x+1)-y_{H}(x)=y(x+)-y(x)-(y_{0}(x+1)-y_{0}(x))=g(x+1)-bg(x)=0.$$
By Theorem \ref{prop:first-order}, we may write $y_{H}(x)$ in the
form
$$y_{H}(x)=\sum_{j=1}^{p}\alpha_{j}e_{a}(x-\xi_{j}),$$
and we obtain the representation
$$y(x)=\sum_{j=1}^{p}\alpha_{j}e_{a}(x-\xi_{j})+\sum_{j=1}^{q}\frac{b\beta_{j}}{b-a}e_{b}(x-x_{j}).$$
Recalling the identity $e_{a}(-x)=\frac{1}{a}e_{1/a}(x)$, we finally
get
$$y(x)=\sum_{j=1}^{p}\alpha_{j}e_{a}(x-\xi_{j})+\sum_{j=1}^{q}\frac{b\beta_{j}}{b-a}e_{b}(-x+x_{j}).$$
To complete the proof, it remains to observe that all tropical
meromorphic functions of type
$$y_{1}(x)=e_{a}(x-s), \qquad y_{2}(x)=e_{a}(-x+t)$$
are solutions to equation (\ref{eqn:3}) as well.
\end{proof}

We now proceed to considering equation~(\ref{eqn:3}) in the case
when $0<|c|<2$. Then the two roots $\lambda_{\pm}$ are complex
conjugates, so that we may put $\lambda_{+}=re^{i\theta}$ and
$\lambda_{-}=r^{-1}e^{-i\theta}$ with real constants $r(\geq 1)$ and
$\theta\in[0,2\pi)$. Since $\Im c=
\Im(\lambda_{+}+\lambda_{-})=(r-r^{-1})\sin\theta=0$, we must have
either $r=1$ or $\theta=0, \pi$. But $\theta=0, \pi$ means that
$|c|=2$, thus we must have $r=1$. Therefore $c=2\cos \theta$
$(0<\theta<\pi)$ with $\lambda_{\pm}=e^{\pm i\theta}$.

\medskip

It is a routine computation to show that
$$Y(x):=e^{i\theta [x]}\left(x-[x]+\frac{1}{e^{i\theta}-1}\right)$$
is a formal solution to equation (\ref{eqn:3}). Therefore,
$y_{1}(x):=\Re Y(x)$ and $y_{2}(x):=\Im Y(x)$ also satisfy
(\ref{eqn:3}). By a straightforward computation, we obtain
\begin{equation}\label{y1}
y_{1}(x)=(\cos (\theta [x]))(x-[x])+\frac{(\cos (\theta
[x]))(\cos\theta -1)+\sin (\theta [x])\sin\theta}{2(1-\cos\theta)}
\end{equation}
and
\begin{equation}\label{y2}
y_{2}(x)=(\sin (\theta
[x]))(x-[x])+\frac{(\sin (\theta [x]))(\cos\theta -1)+(\cos (\theta
[x]))\sin\theta}{2(1-\cos\theta)}.
\end{equation}
These functions are tropical meromorphic, provided they are
continuous at each integer $m\in\mathbb{Z}$. As for $y_{1}(x)$, this
follows by verifying that
$$\cos (\theta [x])+\frac{(\cos (\theta [x]))(\cos\theta -1)+(\sin (\theta [x]))\sin\theta}{2(1-\cos\theta)}$$
$$=\frac{(\cos (\theta [x]+\theta))(\cos\theta -1)+(\sin (\theta [x]+\theta))\sin\theta}{2(1-\cos\theta)};$$
this is an elementary computation. The continuity of $y_{2}(x)$
follows in the same manner. Of course, $y_{j}(x-s)$ and
$y_{j}(-x+t)$, $j=1,2$, are tropical meromorphic solutions to
equation (\ref{eqn:3}) as well. Therefore, we have the following

\begin{theorem}\label{exist} Equation (\ref{eqn:3}) admits tropical
meromorphic solutions for all $c\in\mathbb{R}$.
\end{theorem}

\begin{remark} We conjecture that an arbitrary tropical meromorphic solution $y$ to
(\ref{eqn:3}) in the case $|c|<2$ may be represented as a linear
combination of shifts of (\ref{y1}) and (\ref{y2}), similarly as
already done for the case of $|c|>2$ in Theorem
\ref{prop:second}(iii). We have not yet been able to prove this. To
illustrate the situation, first observe that all tropical
meromorphic solutions are $3$-periodic in the case of $c=-1$.
Indeed, combining
$$y(x+1)+y(x-1)+y(x)=0$$
and
$$y(x+2)+y(x)+y(x+1)=0$$
results in $y(x+2)=y(x-1)$. Recall now that in the case $c=-1$ the
basic solutions to (\ref{eqn:c}) obtained from (\ref{y1}) and
(\ref{y2}) above are
\begin{equation}\label{3minus1}
y_{1}(x)=(\cos (\frac{2\pi}{3}[x]))(x-[x])+\frac{\sqrt{3}}{6}\sin
(\frac{2\pi}{3}[x])-\frac{1}{2}\cos (\frac{2\pi}{3}[x])
\end{equation}
and \begin{equation}\label{3minus2} y_{2}(x)=(\sin
(\frac{2\pi}{3}[x]))(x-[x])-\frac{1}{2}\sin
(\frac{2\pi}{3}[x])-\frac{\sqrt{3}}{6}\cos (\frac{2\pi}{3}[x]).
\end{equation}
On the other hand, it is not difficult to verify that a tropical
meromorphic function $y(x)$ defined by its values in the primitive
period interval $[0,3)$ as
\begin{equation}\label{3minussp}
y(x) = \left\{
\begin{array}{cl}
-\Delta(x) & 0\leq x\leq 1/3 \\
0  & 1/3\leq x\leq 2/3 \\
-\Delta(x-2/3) & 2/3\leq x\leq 1 \\
+\Delta(x-1) & 1\leq x\leq 4/3 \\
0 & 4/3\leq x\leq 8/3 \\
+\Delta(x-8/3) & 8/3\leq x<3,
\end{array}
\right. 
\end{equation}
where $\Delta(x):=x$ when $0\leq x\leq 1/6$, and $\Delta(x):=-x+1/3$
when $1/6\leq x\leq 1/3$, is a solution to equation (\ref{eqn:3}) in
the case of $c=-1$. This function has four poles, resp. five roots,
in the period interval $[0,3)$ so that there are three double and
one single pole, resp. two double and three single roots, see the
related depictions of (\ref{3minus1}) (red), (\ref{3minus2}) (green)
and (\ref{3m}) (blue) in Figure $3$ below.
\begin{figure}[h!]
  \begin{center}
    \includegraphics[keepaspectratio]{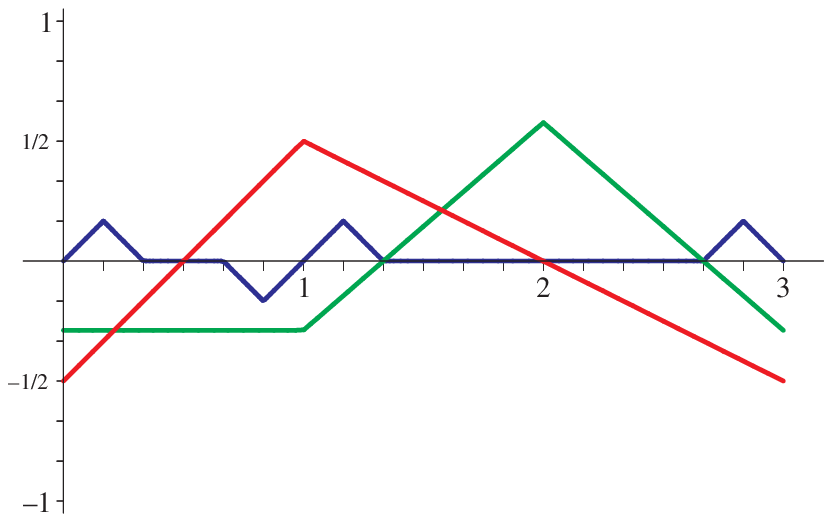}
  \end{center}
 \caption{}
 \label{fig:1}
\end{figure}
We have not been able to determine whether (\ref{3minussp}) could be
represented as a finite linear combination of shifts of
(\ref{3minus1}) and (\ref{3minus2}).
\end{remark}

\begin{remark} We also remark that the solutions (\ref{y1}), (\ref{y2}) to
equation (\ref{eqn:3}) with $c=2\cos\theta$ are tropical periodic,
hence of finite order two, as soon as $\theta$ is of the form
$\theta=2\pi /r$ with $r\in\mathbb{Q}$, while if $\theta$ is an
irrational multiple of $2\pi$, then these solutions are
non-periodic. However, their order is equal to two, see Theorem
\ref{order2} below.
\end{remark}

As a preparation to our final theorem, we add the following
observation. Let $y(x)$ be an arbitrary tropical meromorphic
solution of

\begin{equation}\label{theta}
y(x+1)+y(x-1)=2(\cos\theta)y(x)
\end{equation}

with $\theta\in (0,\pi)$. In this case, the solutions $a,b$ of
$\lambda^{2}-(2\cos\theta)\lambda +1=0$ are $\cos\theta\pm
i\sin\theta$, and so equation (\ref{3m}) now takes the form
$$y(x+1)-(\cos\theta +i\sin\theta)y(x)=(\cos\theta -i\sin\theta)(y(x)-(\cos\theta +i\sin\theta)y(x-1)).$$
Denoting $g(x)=s(x)+it(x):=y(x)-(\cos\theta +i\sin\theta)y(x-1)$, we
see that $g(x)$ is $1$-periodic in the sense of $g(x+1)=(\cos\theta
-i\sin\theta)g(x)$. A simple computation now results in

\begin{equation}\label{matrix}
\left(
    \begin{array}{c}
      s(x+1) \\
      t(x+1) \\
    \end{array}\right)=\left(
                   \begin{array}{cc}
                     \cos\theta & \sin\theta \\
                     -\sin\theta & \cos\theta \\
                   \end{array}
                 \right)\left(
                          \begin{array}{c}
                            s(x) \\
                            t(x) \\
                          \end{array}
                        \right).
\end{equation}

This means, in particular, that the shift $(s(x),t(x))\rightarrow
(s(x+1),t(x+1))$ is a rotation, and since the image of
$(s[0,1],t[0,1])$ is bounded, the vector $(s(x),t(x))$ remains
bounded over the real axis $\mathbb{R}$. Since $y(x)$ is real, we
immediately see that
\begin{equation}\label{yst}
s(x)=y(x)-(\cos\theta)y(x-1), \qquad t(x)=(\sin\theta)y(x-1),
\end{equation}
i.e.
$$\left(
    \begin{array}{c}
      s(x) \\
      t(x) \\
    \end{array}
  \right)=\left(
            \begin{array}{cc}
              1& -\cos\theta \\
              0 & -\sin\theta \\
            \end{array}
          \right)\left(
                   \begin{array}{c}
                     y(x) \\
                     y(x-1) \\
                   \end{array}
                 \right).
$$
Therefore, $y(x)$ is a bounded tropical meromorphic function.

\begin{theorem}\label{order2} Let $y(x)$ be a non-trivial tropical meromorphic
solution of equation
$$y(x+1)+y(x-1)=cy(x).$$
If $|c|>2$, then $y$ is of hyper-order $\rho_{2}(y)=1$, while if
$|c|\leq 2$, then $y$ is of order $\rho (y)=2$.
\end{theorem}

\begin{proof} Suppose first that $|c|\leq 2$. By Theorem
\ref{prop:second}(i)(ii), we may assume that $|c|<2$.
Differentiating (\ref{matrix}), we observe that the slope of $y(x)$
remains uniformly bounded in $\mathbb{R}$. Therefore, the
multiplicities of poles of $y(x)$ are uniformly bounded as well. By
(\ref{yst}) and (\ref{matrix}), we conclude that the number of
distinct slope discontinuities is the same in each interval
$[n,n+1)$ as in the initial interval $[0,1)$. Therefore,
$N(r,y)\asymp \kappa r^{2}$ for some $\kappa
>0$, and since $y(x)$ is bounded, we get $\rho_{2}(y)=2$.

\medskip

Suppose next that $|c|>2$. Recall the proof of Theorem
\ref{prop:second}, see (\ref{grep}) and (\ref{gequ}) there, and fix
$a>1,b<1$ as in this proof. If the hyper-order $\rho_{2}(g)\geq 1$,
then (\ref{gequ}) readily implies that $\rho_{2}(y)\geq 1$. On the
other hand, by the representation of $y$ in Theorem
\ref{prop:second}(iii), and the fact that $\rho_{2}(e_{a}(x))=1$, we
get $\rho_{2}(y)\leq 1$, hence $\rho_{2}(y)=1$. To prove that
$$g(x):=\sum_{j=1}^{q}\beta_{j}e_{b}(x-x_{j})$$
is of hyper-order one, observe that $e_{b}(x-x_{j})$ has no poles
and has zeros exactly at $x_{j}+k$, $k\in\mathbb{Z}$, of
multiplicity $(1-b)\mathbb{b}^{k-1}$. Since the points $x_{j}$ are
distinct, we may fix one of $e_{b}(x-x_{1}),\ldots ,e_{b}(x-x_{q})$,
say $e_{b}(x-x_{s})$, to conclude that
$$2T(r,g)\geq N(r,g)+N(r,1_{\circ}\oslash g)\geq N(r,1_{\circ}\oslash e_{b}(x-x_{s}))\geq K(1/b)^{r}$$
for some positive constant $K$, completing the proof.
\end{proof}

\textbf{Remark.} In a similar way as made above for (\ref{eqn:3}),
we could also treat the slightly more general equation
$$
y(x+1)\otimes y(x-1)^{\otimes d} =y(x)^{\otimes c} \quad \text{i.e.}
\quad y(x+1)-c\, y(x)+d\, y(x-1)=0
$$
for $c,d\in\mathbb{R}$ with $c^2> d$. In fact, the characteristic
equation $\rho^2-c\rho(x)+d=0$ has two real roots $\lambda_1$ and
$\lambda_2$ with $\lambda_1+\lambda_2=c$ and $\lambda_1\lambda_2=d$.
We omit these considerations.

\section{Discussion}\label{disc}

We remark that Sections \ref{Nlinna} to \ref{periodic} open up a
number of possibilities for further investigations such as possible
tropical counterparts to deficiencies and ramifications, value
distribution theory of tropical difference polynomials and
uniqueness theory of meromorphic functions, among other issues.

\medskip

Concerning our applications to ultra-discrete equations, Section
\ref{infinite} and Section \ref{second}, have been restricted,
essentially, to a couple of specific examples only, avoiding linear
ultra-discrete equations that contain tropical addition, as well as
non-homogeneous linear equations and tropical non-linear equations.

\medskip

The topics pointed out here will be treated, we hope, in subsequent
papers.

\bigskip

\textbf{Acknowledgment.} The first author is grateful for a
possibility to visit Kanazawa University twice in 2009. These short
periods enabled us to start preparing and to finishing this paper.


\begin{thebibliography}{99}

\bibitem{CYe} Cherry, W. and Ye, Z., \emph{Nevanlinna's Theory of Value
Distribution}, Springer--Verlag, Berlin, 2001.

\bibitem{CF} Chiang, Y.-M. and Feng, S.-J., \emph{On the Nevanlinna characteristic of
}$f(z+\eta )$ \emph{and difference equations in the complex plane},
Ramanujan J. \textbf{16} (2008), 105--129.

\bibitem{C} Clunie, J., \emph{On integral and meromorphic
functions}, J. London Math. Soc. \textbf{37} (1962), 17--27.

\bibitem{E} Elaydi, S., \emph{An Introduction to Difference
Equations}, Second Ed., Springer-Verlag, New York, 1999.

\bibitem{HK} Halburd, R.G. and Korhonen, R., \emph{Difference analogue of the lemma of the logarithmic
derivative with applications to difference equations}, J. Math.
Anal. Appl. \textbf{314} (2006), 477--487.

\bibitem{HK2} Halburd, R.G. and Korhonen, R., \emph{Finite-order meromorphic solutions and the
discrete Painlev\'{e} equations}, Proc. London Math. Soc.
\textbf{94} (2007), 443--474.

\bibitem{HKT} Halburd, R.G., Korhonen, R. and Tohge, K., \emph{Holomorphic curves with shift-invariant
hyperplane preimages}, to appear. See also arXiv 0903.3236.

\bibitem{HS} Halburd, R.G. and Southall, N., \emph{Tropical Nevanlinna theory and ultra-discrete
equations}, Int. Math. Res. Notices. \textbf{2009} (2009), 887--911.

\bibitem{H} Hayman, W. K., \emph{Meromorphic Functions}, Clarendon Press, Oxford, 1964.

\bibitem{K} Korhonen, R., \emph{A new Clunie type theorem for difference polynomials},
to appear in J. Difference Equ. See also arXiv 0903.4394.

\bibitem{L} Laine, I., \emph{Nevanlinna Theory and Complex Differential
Equations}, Walter de Gruyter, Berlin--New York, 1993.

\bibitem{LY} Laine, I. and Yang, C.-C., \emph{Clunie theorems for difference and q-difference
polynomials}, J. Lond. Math. Soc. \textbf{76} (2007), 556--566.

\bibitem{LY2} Laine, I. and Yang, C.-C., \emph{Tropical versions of Clunie and Mohon'ko
lemmas}, to appear in Complex Variables Elliptic Equ.

\bibitem{M1} Mohon'ko, A., \emph{The Nevanlinna characteristics of certain meromorphic
functions}, Teor. Funktsi\v{i} Funktional. Anal. i Prilozhen
\textbf{14} (1971), 83--87. (Russian)

\bibitem{MM} Mohon'ko, A. and Mohon'ko, V., \emph{Estimates of the Nevanlinna characteristics of certain classes
of meromorphic functions, and their applications to differential
equations}, Sibirsk. Mat. Zh. \textbf{15} (1974), 1305--1322.
(Russian)

\bibitem{SS} Speyer, D. and Sturmfels, B., \emph{Tropical
mathematics}, see arXiv 0408099.

\bibitem{YY} Yang, C.-C. and Ye, Z., \emph{Estimates of the proximate function of differential
polynomials}, Proc. Japan Acad. Ser. A. Math. Sci. \textbf{83}
(2007), 50--55.

\end{thebibliography}
\end{document}